\documentclass[10pt,onecolumn,draftclsnofoot]{IEEEtran}

\usepackage{lipsum}
\usepackage{cite}
\usepackage{amsmath,amssymb,amsfonts}
\usepackage{graphicx}
\usepackage{algorithm,algorithmic}
\usepackage{hyperref}
\hypersetup{hidelinks=true}
\usepackage{textcomp}
\makeatother

\usepackage{graphicx}
\usepackage{amsthm}
\usepackage{tabu}
\usepackage{breqn}
\usepackage{verbatim}
\usepackage{caption}
\usepackage{subcaption}
\usepackage{esvect}
\usepackage{footnote}
\usepackage{siunitx}
\usepackage{color}
\usepackage{url}
\usepackage[makeroom]{cancel}
\usepackage{todonotes}
\newtheorem{theorem}{Theorem}[section]
\newtheorem{corollary}{Corollary}[section]
\newtheorem{definition}{Definition}[section]
\newtheorem{lemma}{Lemma}[section]
\newtheorem{proposition}{Proposition}[section]
\newtheorem{remark}{Remark}
\newtheorem*{problem statement}{Problem Statement}

\newtheorem{assumption}{Assumption}

\newcommand{\real}{\mathbb{R}}

\newcommand{\s}{\boldsymbol{s}}

\newcommand{\lie}{\mathcal{L}}

\newcommand{\Sset}{\mathcal{S}}
\newcommand{\Zset}{\mathcal{Z}}
\newcommand{\normaldi}{\mathrm{N}}
\newcommand{\ball}{\mathbb{B}}

\newcommand{\proj}{\mathrm{proj}}

\DeclareMathOperator*{\argmin}{arg\,min}

\begin{document}
\title{Continuous Approximations of Projected Dynamical Systems via Control Barrier Functions}
\author{Giannis~Delimpaltadakis, Jorge~Cortés, and W.P.M.H.~Heemels
\thanks{Giannis Delimpaltadakis and W.P.M.H. Heemels are with the Control Systems Technology section, Mechanical Engineering, Eindhoven University of Technology, the Netherlands. Jorge Cortés is with the Department of Mechanical and Aerospace
Engineering, University of California San Diego, USA. Emails: \texttt{i.delimpaltadakis@tue.nl, cortes@ucsd.edu, w.p.m.h.heemels@tue.nl}. \newline \indent This research is partially funded by the European Research Council (ERC) under the Advanced ERC grant PROACTHIS, no. 101055384.}}

\maketitle
\begin{abstract}
   \emph{Projected Dynamical Systems} (PDSs) form a class of discontinuous constrained dynamical systems, and have been used widely to solve optimization problems and variational inequalities. Recently, they have also gained significant attention for control purposes, such as high-performance integrators, saturated control and feedback optimization. In this work, we establish that locally Lipschitz continuous dynamics, involving \emph{Control Barrier Functions} (CBFs), namely \emph{CBF-based dynamics}, approximate PDSs. Specifically, we prove that trajectories of CBF-based dynamics \emph{uniformly converge} to trajectories of PDSs, as a CBF-parameter approaches infinity. Towards this, we also prove that CBF-based dynamics are perturbations of PDSs, with quantitative bounds on the perturbation. Our results pave the way to implement discontinuous PDS-based controllers in a continuous fashion, employing CBFs. We demonstrate this on numerical examples on feedback optimization and synchronverter control. Moreover, our results can be employed to numerically simulate PDSs, overcoming disadvantages of existing discretization schemes, such as computing projections to possibly non-convex sets. Finally, this bridge between CBFs and PDSs may yield other potential benefits, including novel insights on stability.
\end{abstract}

\section{Introduction}
\emph{Projected Dynamical Systems} (PDSs) form a class of discontinuous dynamical systems, with trajectories constrained in some set $\Sset$ (see \cite{nagurney1995projected}). In particular, while the state of the system lies in the interior of $\Sset$, it evolves according to some nominal vector-field, but when it reaches the boundary of $\Sset$, the vector field is modified in a discontinuous manner, such that it keeps the trajectory inside $\Sset$. PDSs have been extensively employed for analyzing constrained optimization problems and variational inequalities \cite{nagurney1995projected,brogliato2006equivalence,heemels2000complementarity,hauswirth2021projected}, with applications in price markets, traffic networks, power grids, etc. Recently, they have also attracted considerable interest for control purposes, such as high-performance hybrid integrators called HIGS \cite{deenen2021projection,heemels2023existence}, passivity-based control \cite{chu2023projection}, saturated control \cite{lorenzetti2022pi,fu2023novel} and feedback optimization\footnote{Driving the system's state to the solution of an optimization problem.} \cite{hauswirth2020anti,hauswirth2021optimization,GB-JC-JIP-EDA:22-tcns}. In all these control settings, a dynamic controller with PDS-like dynamics is interconnected to a dynamical system.

\emph{Control Barrier Functions} (CBFs) refer to a set of methods to control systems such that their trajectories remain in a safe set $\Sset$ \cite{ames2019control_review}. In contrast to PDSs, here the nominal dynamics is modified proactively, even before reaching the boundary of $\Sset$. This way, in contrast to PDSs, CBF-based control results in a locally Lipschitz continuous closed-loop vector field. The ``safe" dynamics is specified as the solution to a Quadratic Program (QP), with its constraint (the \emph{CBF constraint}) preventing the system to leave $\Sset$. The CBF constraint involves a tunable parameter, which, loosely, determines how far from the boundary of $\Sset$ the nominal vector-field starts being modified. CBFs have had numerous successful applications in safety-critical scenarios, such as robotics and automotive \cite{ames2019control_review,ames2016control,panagou2015distributed,reis2020control,wences2022compatibility}.

In this work, we reveal an intimate relationship between PDSs and CBFs: we show that PDSs can be approximated by (locally Lipschitz) continuous dynamics, employing CBFs. Specifically, we present results formally establishing the fact that system interconnections containing CBF-based continuous dynamics (hereby termed \emph{CBF-based dynamics}) approximate interconnections with discontinuous PDSs (hereby termed \emph{PDS-based dynamics}). Our results pave the way for implementing the PDS-based feedback controllers mentioned above in a continuous fashion, employing CBFs, with all the advantages that may come with continuity, such as increased robustness and easier sampled-data implementations. 

Moreover, our work has implications for purposes other than PDS-based control as well. First, one may employ locally Lipschitz CBF-based dynamics to numerically simulate discontinuous PDSs, e.g. for constrained optimization. While simulating Lipschitz CBF-based dynamics is computationally easy, standard PDS discretization schemes (see e.g., \cite{nagurney1995projected,edmond2006bv}) involve computing a projection to $\Sset$, which is computationally expensive, especially when $\Sset$ is non-convex. This may be particularly useful for online optimization (e.g., model predictive control). Further, our results build bridges between PDS theory and CBF theory. E.g., in our preliminary work \cite{delimpaltadakis2023relationship}, we showed how, in simple scenarios, one may infer stability of safety-critical CBF-based control from the associated PDS, addressing the notorious undesired equilibrium problem of CBFs (cf. \cite{reis2020control}).

\subsection*{Contributions}
We prove that trajectories of CBF-based dynamics uniformly converge to trajectories of PDS-based dynamics, as the tunable parameter $\alpha$ in the CBF constraint approaches $\infty$. Thus, as $\alpha$ becomes larger, CBF-based dynamics approximate more accurately PDS-based dynamics. This supports the use of CBF-based dynamics to implement PDS-based dynamics, as, for large $\alpha$, their trajectories are arbitrarily close. Towards proving convergence of trajectories, we also show that CBF-based dynamics are perturbations of PDSs, in the sense that the CBF-based vector field belongs in a set of perturbations of the set-valued map of the differential inclusion that describes the PDS. We obtain quantitative bounds on the perturbation, that depend on and vanish with $\alpha$. Finally, we demonstrate our results on examples on feedback optimization and synchronverter control, implementing PDS-based controllers from the literature in a continuous fashion, via CBFs.

\subsection*{Related work}
Our preliminary work \cite{delimpaltadakis2023relationship} explored the relationship between PDSs and CBFs. However, \cite{delimpaltadakis2023relationship} did not study convergence of trajectories, which is the present article's main theme.\footnote{As such, Cors. \ref{cor:interconnected_cbf_solutions_to_perturbed_DI}, \ref{cor:cbf_perturbation_of_limit_di} and \ref{cor:pds_same_solutions_with_limit_di}, Thms. \ref{thm:solution_convergence} and \ref{thm:stability_solution_convergence}, and Prop. \ref{prop:well-posed-limit-DI} appear here for the first time.} Further, \cite{delimpaltadakis2023relationship}, focusing on autonomous dynamics, proved that CBF-based dynamics are perturbed versions of PDSs \cite[Thm. III.1]{delimpaltadakis2023relationship}. Here, we extend \cite[Thm. III.1]{delimpaltadakis2023relationship} in two ways: a) we consider the more general non-autonomous case of interconnections of systems with dynamics that are either PDS-like or CBF-based, and b) we provide bounds on the perturbed set-valued map, establishing local boundedness thereof, which is necessary for proving convergence of trajectories.

The works~\cite{allibhoy2023control,allibhoy2023anytime,chen2023online} also employ CBF-based dynamics instead of PDS-based ones for constrained optimization problems, variational inequalities and feedback optimization. Mainly, they focus on the relationship between the \emph{equilibria} of CBF-based and PDS-based dynamics, and provide asymptotic stability and contractivity results on CBF-based dynamics. Further, in \cite{allibhoy2023control}, it is proven that the right-hand sides of the differential equations of PDSs and CBF-based dynamics coincide at $\alpha=\infty$. However, neither convergence of trajectories nor the fact that CBF-based dynamics are perturbations of PDS-based ones, with quantitative bounds on the perturbation, depending on $\alpha$, is established. Overall, our work can be viewed as \emph{complementary} to \cite{allibhoy2023control,allibhoy2023anytime,chen2023online}.

Finally, \cite{hauswirth2020anti} proposed a type of continuous approximations of PDSs, based on antiwindup control, called Anti-Windup Approximations (AWA). Some steps on establishing convergence of trajectories in our work are inspired by \cite{hauswirth2020anti} (e.g., proving that CBF-based dynamics are perturbations of PDS-based ones). Nonetheless, as CBF-based dynamics are different than AWA, proving the corresponding results is also significantly different. Further, trajectories of AWA, contrary to CBF-based dynamics, stay in an inflation of $\Sset$. If one needs to force AWA to stay in $\Sset$, modifications have to be made, which may result in undesired behaviors, such as discarding useful equilibria (e.g., extrema of feedback optimization) or shrinking regions of attraction. In addition, AWA involve calculating projections to $\Sset$, which can be computationally expensive. On the other hand, when $\Sset$ represents limitations of the physical system (e.g., input saturation), and said projections are physically enforced by the system, implementation of AWA does not even require knowledge of $\Sset$, in contrast to CBF-based dynamics. Apart from the above, both our work and \cite{hauswirth2020anti} establish interesting, previously unknown, connections between PDSs and other popular classes of dynamics, and might even suggest connections between antiwindup control and CBFs.

\section{Preliminaries}
\subsection{Notation}
Given a closed set $\Sset\subseteq\real^n$, denote by $\partial\Sset$ and $\mathrm{Int}(\Sset)$ its boundary and interior, respectively. Given $x\in\real^n$, its Euclidean distance to $\Sset$ is $d(x,\Sset):= \min_{y\in\Sset}\|x-y\|$, where $\|\cdot\|$ denotes the Euclidean norm. Its projection to $\Sset$ is $\proj_{\Sset}(x) :=\argmin_{y\in\Sset}\|x-y\|$. Given a function $f:X\to X'$, $\mathrm{dom}f$ denotes its domain.

Denote the set of positive-definite symmetric matrices in $\real^{n\times n}$ by $\mathbb{S}_+^n$. Given $P\in\mathbb{S}_+^n$, denote its minimum and maximum eigenvalues by $\underline{\lambda}(P)$ and $\overline{\lambda}(P)$, respectively. Given $P\in\mathbb{S}_+^n$ and $x\in\real^n$, denote $\|x\|_P=\sqrt{x^\top P x}$. Given a function $f:\real^m\times\real^n\to\real^{n}$ and a differentiable function $h:\real^n\to\real$, denote $\lie_fh(z,x):=\nabla h^\top(x)\cdot f(z,x)$, where $(z,x)\in\real^m\times\real^n$. A continuous function $\gamma:[0,\infty)\to\real$ is said to belong to $\mathcal{K}_\infty$, if $\gamma(0)=0$, $\gamma$ is strictly increasing, and $\lim_{a\to\infty}\gamma(a)=+\infty$. Finally, denote the closed unit ball of appropriate dimensions by $\ball$.

\subsection{Variational Analysis}
Here, we recall some basic concepts from variational analysis. For more detail, the reader is referred to \cite{rockafellar2009variational}.
\begin{definition}[Tangent Cone {\cite[Def. 6.1]{rockafellar2009variational}}]
    Given a closed set $\Sset\subseteq\real^n$, the \emph{tangent cone} to $\Sset$ at $x\in\Sset$, denoted by $T_\Sset(x)$, is the set of all vectors $w \in\real^n$ for which there exist sequences $\{x_i\}_{i\in\mathbb{N}}\in\Sset$ and $\{t_i\}_{i\in\mathbb{N}}$, $t_i>0$, with $x_i\to x$ and $t_i\to 0$, such that $w=\lim_{i\to\infty}\frac{x_i-x}{t_i}$.
\end{definition}
Intuitively, $T_{\Sset}(x)$ is the set of all vectors to the direction of which if we move infinitesimally from $x$, we still remain in $\Sset$.
\begin{definition}[Clarke regularity {\cite[Cor. 6.29]{rockafellar2009variational}}]
    Given a closed set $\Sset\subseteq\real^n$, $\Sset$ is \emph{Clarke regular} if and only if the set-valued map $x\mapsto T_\Sset(x)$ is inner semicontinuous\footnote{For the definition of inner semicontinuity, see \cite[Def. 5.4]{rockafellar2009variational}.}.
\end{definition}
As, in this paper, we work with sets that are at least Clarke regular, we employ the following definition of \emph{normal cone} and need not distinguish between different notions thereof:
\begin{definition}[Normal Cone {\cite[Prop. 6.5]{rockafellar2009variational}}]
    Given a closed set $\Sset\subseteq\real^n$, define the \emph{normal cone} of $\Sset$ at $x\in\Sset$ as $N_{\Sset}(x):=\{\eta\in\real^n\mid\eta^\top v\leq0, \text{ }\forall v\in T_\Sset(x)\}$.
\end{definition}
\begin{definition}[Prox-regularity {\cite[Def. 2.1]{adly2016preservation}}]
    A Clarke regular set $\Sset$ is \emph{prox-regular}, if there exists $\gamma>0$ such that, for any $x\in\Sset$: $\eta^\top(y-x)\leq\gamma\|\eta\|\|y-x\|^2$, for all $\eta\in N_{\Sset}(x)$ and $y\in\Sset$.
\end{definition}
Sets of the form 
\begin{equation}\label{eq:sset}
    \Sset = \{x\in\real^n\mid h(x)\geq 0\}
\end{equation}
where $h:\real^n\to\real$ is continuously differentiable, $x\mapsto\nabla h(x)$ is locally Lipschitz, and $\nabla h(x)\neq 0$ for all $x\in\partial\Sset$, are prox-regular and the expressions for their tangent and normal cones are obtained as follows (see \cite[Example 6.8]{hauswirth2021projected}), for any $x\in\Sset$:
\begin{equation}\label{eq:cones_proxregular}
\begin{aligned}
    T_{\Sset}(x) &= \left\{\begin{aligned}
        &\real^n, &\qquad x\in\mathrm{Int}(\Sset)\\
        &\{v\in\real^n\mid \nabla^\top h(x)\cdot v \geq 0\}, & \qquad x\in\partial\Sset
        \end{aligned}\right.\\
    N_{\Sset}(x)&=\left\{\begin{aligned}
        &\{0\}, &x\in\mathrm{Int}(\Sset)\\
        &\{\eta\in\real^n\mid \eta=\lambda\nabla h(x),\text{ }\lambda\leq0\}, &x\in\partial\Sset
        \end{aligned}\right.
\end{aligned}
\end{equation} 
\subsection{Differential Inclusions}
Given $\Sset\subseteq\real^n$, consider the constrained \emph{differential inclusion} (DI)
\begin{equation}\label{eq:DI}
    \dot{\xi} \in F(\xi), \quad \xi \in \Sset
\end{equation}
where $F:\real^n\rightrightarrows\real^n$ and we have omitted the dependence of $\xi$ on $t$, for conciseness. When it is clear from the context that the DI's constraint set is $\real^n$, we might omit it. Given an initial condition $x_0\in\Sset$ and a $T\in[0,\infty)$, a function $\phi:[0,T]:\to\real^n$ is a (Carathéodory) solution to DI \eqref{eq:DI}, if it is absolutely continuous, $\phi(0)=x_0$, $\phi(t)\in\Sset$ and $\dot{\phi}(t)\in F(\phi(t))$ for almost all $t\in[0,T]$. We adopt the same definition for solutions of differential equations, as well.
\begin{definition}[$\sigma$-perturbation {\cite[Def. 6.27 simplified]{hybrid_book}}]\label{def:sigma-perturbation}
    The \emph{$\sigma$-perturbation} of DI \eqref{eq:DI} is defined as
    \begin{equation*}
        \dot{\xi} \in F_\sigma(\xi), \quad \xi \in \Sset_\sigma
    \end{equation*}
    where $F_\sigma(x) = \overline{\mathrm{con}}F\Big((x+\sigma\ball)\cap\Sset\Big)+\sigma\ball$, $\overline{\mathrm{con}}$ denotes convex closure and $\Sset_\sigma = \Sset +\sigma\ball$.
\end{definition}
\section{Background on Projected Dynamical Systems and Control Barrier Functions}\label{sec:background_cbfs_pdss}
Both PDSs and CBF-based control methods start from an unconstrained nominal system 
\begin{equation}\label{eq:sys_nominal}
    \dot{\xi} = f(\xi)
\end{equation}
where $f:\real^n\to\real^n$, and end up with one that is constrained in some set $\Sset\subseteq\real^n$. As commonly done in the literature \cite{ames2019control_review,ames2016control,panagou2015distributed}, we consider sets $\Sset$ of the form \eqref{eq:sset}. For most of the results stated in this article, the following set of assumptions on $\Sset$ and $h$ is employed:
\begin{assumption}\label{assum:sset_and_h} $\Sset$ and $h:\real^n\to\real$ satisfy the following:
\begin{enumerate}
    \item \label{assum:compactness} $\Sset$ is nonempty and compact.
    \item \label{assum:h_C11} $h$ is continuously differentiable, $x\mapsto \nabla h(x)$ is locally Lipschitz, and its Lipschitz constant on $\Sset$ is $L_{\nabla h}$.
    \item \label{assum:regularity} For all $x\in\real^n$ such that $h(x)=0$, we have  $\nabla h(x)\neq0$.
    \item \label{assum:Kinf} There exists $\gamma\in\mathcal{K}_\infty$, such that, for all $x\in\Sset$: $d(x,\partial\Sset)\leq\gamma(h(x))$.
\end{enumerate}
\end{assumption}
Items \ref{assum:h_C11}, \ref{assum:regularity} and item \ref{assum:compactness} without the compactness assumption, are standard in the literature of CBFs (see e.g., \cite{ames2019control_review} or \cite{wences2022compatibility}). Compactness of $\Sset$ is needed for several bounds in the proof of Prop. \ref{prop:perturbation}. Nonetheless, in many related applications, such as PDS-based control, this is not a restrictive assumption. In fact, as $\Sset$ can be taken arbitrarily large, our results are semiglobal. Further, item \ref{assum:Kinf} holds a-priori if $h$ is a real-analytic function, by the \emph{Łojasiewicz inequality} (see, e.g.,\cite{tiep2012lojasiewicz}). Also, as aforementioned, items \ref{assum:h_C11} and \ref{assum:regularity} imply that $\Sset$ is prox-regular and its tangent and normal cones are expressed as in \eqref{eq:cones_proxregular}. Finally, the following standard assumption on $f$ is considered:
\begin{assumption}\label{assum:f_lipschitz}
$f$ is locally Lipschitz.
\end{assumption}

\subsection{Projected Dynamical Systems}
Given a matrix $P\in\mathbb{S}_+^{n}$, for any two vectors $x,v\in\real^n$, we define the projection operator $\Pi_{\Sset}^P:\real^n\times\real^n\rightrightarrows\real^n$ by
\begin{equation*}
    \Pi_{\Sset}^P(x,v) = \argmin_{\mu\in T_{\Sset}(x)} \|\mu-v\|^2_P
\end{equation*}
Now, consider the following DI:
\begin{equation}\label{eq:sys_pds}
\dot{\xi} \in \Pi^P_{\Sset}\Big(\xi,f(\xi)\Big) = \argmin_{\mu\in T_{\Sset}(\xi)} \|\mu-f(\xi)\|^2_P
\end{equation}
Systems of the form \eqref{eq:sys_pds} are called \emph{Projected Dynamical Systems} (PDSs) \cite{nagurney1995projected}. When $\xi\in\mathrm{Int}(\Sset)$, we have $T_{\Sset}(\xi)=\real^n$, and thus $\Pi^P_{\Sset}\Big(\xi,f(\xi)\Big) = f(\xi)$, i.e., the PDS \eqref{eq:sys_pds} evolves according to the unconstrained dynamics \eqref{eq:sys_nominal}. However, when $\xi\in\partial\Sset$, a vector $\mu\in T_{\Sset}(\xi)$ is chosen (minimizing $\|\mu-f(\xi)\|_P^2$), so that the trajectory remains in $\Sset$. Finally, when $\Sset$ is prox-regular, the right-hand side of \eqref{eq:sys_pds} is a singleton and \eqref{eq:sys_pds} becomes a discontinuous ODE.

Given a set $\Sset$, a matrix $P\in\mathbb{S}^n_+$, two vectors $x,v\in\real^n$ and a scalar $d\in[0,\infty]$ define the set-valued map $\normaldi^P_\Sset:\real^n\times\real^n\times[0,\infty]\rightrightarrows \real^n$ by
\begin{equation*}
    \normaldi^P_\Sset(x,v,d) := v-\Big(P^{-1}N_\Sset(x)\cap d\ball\Big)  
\end{equation*}
Given a function $d:\real^n\to[0,\infty]$, under Assumptions \ref{assum:sset_and_h} and \ref{assum:f_lipschitz}, and further assumptions on $d$, solutions of the PDS \eqref{eq:sys_pds} are equivalent to the solutions of the following DI (see Thm. \ref{thm:hauswirth_solutions}, below)
\begin{equation}\label{eq:pds_di}
    \dot{\xi}\in \normaldi^P_\Sset \Big(\xi,f(\xi), d(\xi)\Big),\quad \xi\in\Sset
\end{equation}
The interpretation of \eqref{eq:pds_di} is: when $\xi(t)$ is about to leave $\Sset$ (i.e., $\xi(t)\in\partial\Sset$), an element $\eta$ from the truncated normal cone $P^{-1}N_\Sset(\xi)\cap d(\xi)\ball$ is chosen, such that $f(\xi)-\eta$ points inwards $\Sset$, thereby keeping $\xi(t)\in\Sset$. The following is a collection of results on solutions of PDSs from \cite{hauswirth2021projected,hauswirth2020anti}:
\begin{theorem}[\hspace{-.1mm}\cite{hauswirth2021projected,hauswirth2020anti}]\label{thm:hauswirth_solutions}
    The following hold:
    \begin{enumerate}
        \item Let $\Sset$ be Clarke regular and $f$ be continuous. Then, \eqref{eq:sys_pds} admits a solution for every initial condition in $\Sset$.
        \item Let $\Sset$ be prox-regular and $f$ be locally Lipschitz. Then, \eqref{eq:sys_pds} admits a unique solution for every initial condition in $\Sset$.
        \item Let $\Sset$ be Clarke regular and $f$ be continuous. Consider $d:\real^n\to[0,\infty]$, such that $d(x)\geq \s\sqrt{\tfrac{\overline{\lambda}(P)}{\underline{\lambda}(P)}}\|f(x)\|$ for all $x\in\Sset$. Then, for every initial condition in $\Sset$, \eqref{eq:sys_pds} and \eqref{eq:pds_di} admit the same solutions.
    \end{enumerate}
\end{theorem}
Notice that prox-regularity implies Clarke regularity and that Assumption \ref{assum:sset_and_h} implies prox-regularity of $\Sset$. Thus, in our context, \eqref{eq:sys_pds} and \eqref{eq:pds_di} share the same unique solution, for every initial condition.

Finally, of particular interest to us are cases where PDSs are interconnected with other systems:
\begin{equation*}
    \begin{aligned}
        \dot{\zeta} = g(\zeta,\xi), \quad
        \dot{\xi} \in \Pi^P_{\Sset}\Big(\xi, f(\zeta,\xi)\Big)
    \end{aligned}
\end{equation*}
where $f:\real^m\times\real^n\to\real^n$ $g:\real^m\times\real^n\to\real^m$. The above setup can have different interpretations. For example, $\zeta(t)$ may be the state of some system controlled by a dynamic controller $\xi(t)$, with PDS dynamics. Such interconnections have been proposed e.g., for saturated control \cite{lorenzetti2022pi,fu2023novel} and feedback optimization \cite{hauswirth2021projected}.\footnote{A generalization of \eqref{eq:interconnection_pds}, termed \emph{extended PDSs}, has also been employed in \cite{deenen2021projection,heemels2023existence} for high-performance integral control and \cite{chu2023projection} for passivity-based control. Extension of our results to ePDSs is considered future work.} Moreover, one can disregard the $\zeta$-part (i.e., $g(\zeta,\xi)=0$) in \eqref{eq:interconnection_pds}, and obtain a standard PDS. Thus, our work also applies to cases such as constrained optimization \cite{nagurney1995projected,brogliato2006equivalence,heemels2000complementarity,hauswirth2021projected}. The above interconnection can be written as:
\begin{equation}\label{eq:interconnection_pds}
    \begin{aligned}
        \begin{bmatrix}
            \dot{\zeta}\\ \dot{\xi}
        \end{bmatrix} \in \Pi^{\tilde{P}}_{\real^m\times\Sset}\bigg((\zeta,\xi), \begin{bmatrix}
            g(\zeta,\xi)\\ f(\zeta,\xi)
        \end{bmatrix}\bigg)
    \end{aligned}
\end{equation}
where $\tilde{P} = \begin{bmatrix}
    I &0\\ 0 &P
\end{bmatrix}$. The normal-cone DI associated to \eqref{eq:interconnection_pds} is:
\begin{equation}\label{eq:interconnection_DI}
\begin{aligned}
    \begin{bmatrix}
            \dot{\zeta}\\ \dot{\xi}
    \end{bmatrix} \in \begin{bmatrix}
            g(\zeta,\xi)\\ \normaldi^P_\Sset\Big(\xi,f(\zeta,\xi),d(\zeta,\xi)\Big)
        \end{bmatrix}, \quad(\zeta,\xi)\in\real^m\times\Sset&   
\end{aligned}
\end{equation}
where $d:\real^n\times\real^m\to[0,\infty]$ is such that $d(z,x)\geq \sqrt{\tfrac{\overline{\lambda}(P)}{\underline{\lambda}(P)}}\|f(z,x)\|$ for all $(z,x)\in\real^m\times\Sset$, and again by Thm. \ref{thm:hauswirth_solutions}, under Assumption \ref{assum:sset_and_h} and local Lipschitz assumptions on $f,g$, it shares the same solutions with the PDS.

\subsection{Control-Barrier-Function-based dynamics}
Given $f:\real^n\to\real^n$, a parameter $\alpha>0$ and $P\in\mathbb{S}_+^{n}$, consider
\begin{equation}\label{eq:sys_cbf}
    \dot{\xi} = f^P_{\mathrm{cbf},\alpha}(\xi) := \left\{\begin{aligned}\argmin_\mu \quad&\|\mu-f(\xi)\|^2_P\\
    \mathrm{s.t.:} \quad &\lie_\mu h(\xi) + \alpha h(\xi) \geq 0\end{aligned}\right. 
\end{equation}
We refer to \eqref{eq:sys_cbf} as the \emph{CBF-based dynamics} as $h$ is referred to as the \emph{control barrier function} (CBF) in the CBF literature \cite{ames2019control_review}. Under Assumption \ref{assum:sset_and_h}, the QP above has a unique solution, for any $\xi$, which can be written in closed form as follows (see \cite[Thm. 1]{wences2022compatibility}):
\begin{equation*}
    f^P_{\mathrm{cbf},\alpha}(\xi) = f(\xi) - \min\Big(0,\lie_f h(\xi) + \alpha h(\xi)\Big)\frac{P^{-1}\nabla h(\xi)}{\|\nabla h(\xi)\|^2_{P^{-1}}}
\end{equation*}
We call $f^P_{\mathrm{cbf},\alpha}$ the \emph{CBF-based vector field}. In contrast to PDSs \eqref{eq:sys_pds}, the nominal dynamics is modified even when $\xi\in\mathrm{Int}\Sset$; $f$ is modified whenever $\lie_f h(\xi) + \alpha h(\xi)\leq 0$, i.e., whenever $h$ decreases along the nominal system's trajectories \eqref{eq:sys_nominal} faster than a state-dependent threshold $-\alpha h(\xi)$. The larger $\alpha$ is, the less invasive the modification is, as it allows for even more negative decreases $\lie_f h(\xi)$. Typically, this allows trajectories to approach closer to the boundary, before the nominal vector-field is modified. 

Vector fields like \eqref{eq:sys_cbf} arise in control systems of the form $\dot{x} = f(x)+u(x)$, where the controller $u(x)$ is designed via CBF-based methods. From standard CBF theory \cite{ames2019control_review}, trajectories of \eqref{eq:sys_cbf} starting in $\Sset$ stay in $\Sset$. Under Assumptions \ref{assum:sset_and_h} and \ref{assum:f_lipschitz}, $f^P_{\mathrm{cbf},\alpha}$ is locally Lipschitz, implying existence and uniqueness of solutions of \eqref{eq:sys_cbf}.

Finally, as with PDSs, we also consider the interconnection
\begin{equation}\label{eq:interconnection_cbf}
    \begin{aligned}
        \dot{\zeta} &= g(\zeta,\xi)\\
        \dot{\xi} &= f^P_{\mathrm{cbf},\alpha}(\zeta,\xi) := \left\{\begin{aligned}\argmin_\mu \quad&\|\mu-f(\zeta,\xi)\|^2_P\\
    \mathrm{s.t.:} \quad &\lie_\mu h(\xi) + \alpha h(\xi) \geq 0\end{aligned}\right.\\
    &= f(\zeta,\xi) - \min\Big(0,\lie_f h(\zeta,\xi) + \alpha h(\xi)\Big)\frac{P^{-1}\nabla h(\xi)}{\|\nabla h(\xi)\|^2_{P^{-1}}}
    \end{aligned}
\end{equation}
where $f:\real^m\times\real^n\to\real^n$ and $g:\real^m\times\real^n\to\real^m$. The article's main objective is to show that the feedback interconnection with a CBF-based controller \eqref{eq:interconnection_cbf} is a continuous approximation of the feedback interconnection with a PDS-based controller \eqref{eq:interconnection_pds}, in the sense of uniform convergence of trajectories. Briefly, as $\alpha\to\infty$, trajectories of \eqref{eq:interconnection_cbf} become identical to trajectories of \eqref{eq:interconnection_pds}. Towards this, we also derive another highly relevant result: that \eqref{eq:interconnection_cbf} is a perturbation of \eqref{eq:interconnection_pds}, with quantitative bounds on the perturbation.

\begin{remark}
{\rm
    Note that all our results relating (trajectories of) \eqref{eq:interconnection_cbf} with (trajectories of) \eqref{eq:interconnection_pds}, also hold for \eqref{eq:sys_cbf} and \eqref{eq:sys_pds}, as \eqref{eq:sys_cbf} and \eqref{eq:sys_pds} are special cases of \eqref{eq:interconnection_cbf} and \eqref{eq:interconnection_pds}, respectively.
    } \hfill $\blacksquare$
\end{remark}

\section{The relationship between PDSs and CBF-based dynamics}
In this section, we delve into the relationship between PDS-based dynamics \eqref{eq:interconnection_DI} and CBF-based dynamics \eqref{eq:interconnection_cbf}. First, we prove that CBF-based dynamics are perturbations of PDSs. Then, building on this, we prove that trajectories of CBF-based dynamics uniformly converge to trajectories of PDSs as $\alpha\to\infty$. The proofs to our results can be found in Section \ref{sec:proofs}, unless otherwise stated.

\subsection{CBF-based dynamics as perturbations of PDSs}
Here, we show how CBF-based dynamics can be viewed as perturbed versions of PDSs, in the sense of Def. \ref{def:sigma-perturbation}. First, as an extension to our preliminary work \cite[Thm. III.1]{delimpaltadakis2023relationship}, we show that $f^P_{\mathrm{cbf},\alpha}(z,x)$ belongs to a set of perturbations of the set-valued map $f(z,x)-P^{-1}N_{\Sset}(x)$.
\begin{proposition}\label{prop:perturbation}
    Consider a vector field $f:\real^m\times\real^n\to\real^n$, a compact set $\Zset\subset\real^m$ and a set $\Sset\subset\real^n$ defined by \eqref{eq:sset}. Let Assumptions \ref{assum:sset_and_h} and \ref{assum:f_lipschitz} hold. Denote the Lipschitz constant of $f$ on $\Zset\times\Sset$ by $L_f$. For $\epsilon$ such that $0<\epsilon<\min_{z\in\partial\Sset}\|\nabla h(z)\|$, define
\begin{equation*}\label{eq:a_1}
        \alpha_* = \frac{\max_{(z,x)\in\Zset\times\Sset}|\lie_f h(z,x)|}{\gamma^{-1}\Big(\frac{\min_{z\in\partial\Sset}\|\nabla h(z)\|- \epsilon}{L_{\nabla h}}\Big)}
\end{equation*}
Moreover, denote
\begin{align*}
    &M_1:= \min_{x\in\partial\Sset}\|\nabla h(x)\|, \quad
    M_2:=\max_{x\in\partial\Sset}\|\nabla h(x)\|, \quad
    M_3:= M_2 + L_{\nabla h}\gamma(\frac{1}{\alpha_*}\max_{(z,x)\in\Zset\times\Sset}|\lie_f h(z,x)|)\\
    &L_1 := \frac{\overline{\lambda}(P)}{\underline{\lambda}(P)\epsilon^2}L_{\nabla h}\biggl[1 + \frac{M_2\overline{\lambda}(P)\Big(M_2+M_3\Big)}{\underline{\lambda}(P)M_1^2}\bigg]
\end{align*}
Finally, define
\begin{align*}
    \sigma(\alpha) := \hspace{-3mm}\max\limits_{(z,x)\in\Zset\times\Sset}\hspace{-2mm}\max&\Big\{\gamma(\frac{1}{\alpha}|\lie_f h(z,x)|), \text{ }\Big(L_f+L_1|\lie_f h(z,x)|\Big)\gamma(\frac{1}{\alpha}|\lie_f h(z,x)|)\Big\}
\end{align*}
and
\begin{align*}
    \delta := \max\limits_{(z,x)\in\Zset\times\Sset}\Big(1+\frac{L_{\nabla h}\gamma(\frac{1}{\alpha_*}|\lie_f h(z,x)|)}{M_1}\Big)\frac{\overline{\lambda}(P)}{\underline{\lambda}(P)}\|f(z,x)\|
\end{align*}
Then, for any $\alpha\geq \alpha_*$, it holds that, for all $(z,x)\in\Zset\times\Sset$:
    \begin{align*}
            f^P_{\mathrm{cbf},\alpha}(z,x)\in K_{\sigma(\alpha)}\bigg(\normaldi^P_\Sset\Big(x,f(z,x),\delta\Big)\bigg)
    \end{align*}
where $f^P_{\mathrm{cbf},\alpha}$ is defined in \eqref{eq:interconnection_cbf} and
\begin{align*}
    &K_{\sigma(\alpha)}\bigg(\normaldi^P_\Sset\Big(x,f(z,x),\delta\Big)\bigg):=
    \bigg\{\normaldi^P_\Sset\Big(y,f(z,y),\delta\Big)+\sigma(\alpha)\ball\text{ }\Big|\text{ } y\in(x+\sigma(\alpha)\ball)\cap\Sset\bigg\}
\end{align*}
\end{proposition}
\begin{remark}\label{rem:differences_with_earlier_work}
{\rm
    Prop. \ref{prop:perturbation} extends \cite[Thm. III.1]{delimpaltadakis2023relationship} in two ways. First, towards addressing interconnections \eqref{eq:interconnection_DI} and \eqref{eq:interconnection_cbf}, it considers $f:\real^m\times\real^n\to\real^n$, in contrast to \cite[Thm. III.1]{delimpaltadakis2023relationship}, which only considers $f:\real^n\to\real^n$. Second, it gives bounds $\delta$ on the normal cone elements in the right-hand side. This is crucial for proving convergence of trajectories later on, as it enables results from \cite{hybrid_book} for \emph{well-posed} DIs, requiring local boundedness of the set-valued map.
    }\hfill $\blacksquare$
\end{remark}
Observe that, due to Assumption \ref{assum:sset_and_h} item \ref{assum:regularity}, $M_1>0$, and, hence, there is always an $\epsilon$ such that $0<\epsilon<\min_{x\in\partial\Sset}\|\nabla h(x)\|$. Moreover, both $\sigma(\alpha)$ and $\delta$ are well-defined, as they are maxima of continuous locally bounded functions over compact sets. The set $\Zset$ constrains the first argument of $f$, i.e., $z\in\Zset\subset\real^m$. We thus study solutions of PDS-based dynamics \eqref{eq:interconnection_pds} and CBF-based dynamics \eqref{eq:interconnection_cbf} that are constrained in $\Zset\times\Sset$, as it becomes clear below (see Remark \ref{rem:sset and zset}).

Prop. \ref{prop:perturbation} gives rise to the two following results:
\begin{corollary}\label{cor:interconnected_cbf_solutions_to_perturbed_DI}
    Given $g:\real^m\times\real^n\to\real^n$ and any $\alpha\geq\alpha_*$, under the assumptions of Prop. \ref{prop:perturbation}, consider the following DI:
    \begin{equation}\label{eq:perturbed_DI}
    \begin{aligned}
        \begin{bmatrix}
                \dot{\zeta}\\ \dot{\xi}
        \end{bmatrix} \in \begin{bmatrix}
                g(\zeta,\xi)\\ K_{\sigma(\alpha)}\bigg(\normaldi^P_\Sset\Big(\xi,f(\zeta,\xi),\delta\Big)\bigg)
            \end{bmatrix},\quad (\zeta,\xi)\in\real^m\times(\Sset+\sigma(\alpha)\ball)&    
    \end{aligned}
    \end{equation}
    Given an initial condition $(z_0,x_0)\in\Zset\times\Sset$ and any $T\in[0,\infty)$, consider $\phi:[0,T]\to\real^m\times\real^n$, such that $\phi(0)=(z_0,x_0)$. Assume that $\phi(t)\in\Zset\times\Sset$ for all $t\in[0,T]$. Then, if $\phi$ is a solution to the CBF-based dynamics \eqref{eq:interconnection_cbf}, it is a solution to DI \eqref{eq:perturbed_DI}.
\end{corollary}
\begin{proof}
    Follows readily from Prop. \ref{prop:perturbation}.
\end{proof}
\begin{corollary}\label{cor:cbf_perturbation_of_limit_di}
    Given $g:\real^m\times\real^n\to\real^n$, under the assumptions of Prop. \ref{prop:perturbation},
    for any $\alpha\geq\alpha_*$, DI \eqref{eq:perturbed_DI} is a $\sigma(\alpha)$-perturbation of the following DI:
    \begin{equation}\label{eq:limit_DI}
    \begin{aligned}
        \begin{bmatrix}
                \dot{\zeta}\\ \dot{\xi}
        \end{bmatrix} \in \begin{bmatrix}
                g(\zeta,\xi)\\ \normaldi^P_\Sset\Big(\xi,f(\zeta,\xi),\delta\Big)
            \end{bmatrix}, \quad (\zeta,\xi)\in\real^m\times\Sset    
    \end{aligned}
    \end{equation}
\end{corollary}
\begin{proof}
    Follows directly by combining Prop. \ref{prop:perturbation} and Def. \ref{def:sigma-perturbation}.
\end{proof}
Moreover, DI \eqref{eq:limit_DI}, under our assumptions, shares the same bounded solutions with PDS \eqref{eq:interconnection_pds}:
\begin{corollary}
\label{cor:pds_same_solutions_with_limit_di}
Given $g:\real^m\times\real^n\to\real^n$, under the assumptions of Prop. \ref{prop:perturbation}, for any $\alpha\geq\alpha_*$, consider PDS \eqref{eq:interconnection_pds} and DI \eqref{eq:limit_DI}. Given an initial condition $(z_0,x_0)\in\Zset\times\Sset$ and any $T\in[0,\infty)$, consider a function $\phi:[0,T]\to\real^m\times\real^n$, such that $\phi(0)=(z_0,x_0)$. Assume that $\phi(t)\in\Zset\times\Sset$ for all $t\in[0,T]$. The following hold:
    \begin{enumerate}
        \item Let $g$ be continuous. Then, $\phi$ is a solution to PDS \eqref{eq:interconnection_pds} if and only if it is a solution to DI \eqref{eq:limit_DI}.
        \item Further, let $g$ be locally Lipschitz. Then, if $\phi$ is a solution to PDS \eqref{eq:interconnection_pds}, and thus to DI \eqref{eq:limit_DI}, it is the only solution.
    \end{enumerate}
\end{corollary}
\begin{proof}
    Observe that, for all $(z,x)\in\Zset\times\Sset$, we have $\delta\geq \sqrt{\tfrac{\overline{\lambda}(P)}{\underline{\lambda}(P)}}\|f(z,x)\|$. The result then follows by Theorem \ref{thm:hauswirth_solutions}.
\end{proof}
Let us explain the significance of Cors. \ref{cor:interconnected_cbf_solutions_to_perturbed_DI}, \ref{cor:cbf_perturbation_of_limit_di} and \ref{cor:pds_same_solutions_with_limit_di}. First, note that solutions of the CBF-based dynamics \eqref{eq:interconnection_cbf} are solutions of perturbed DI \eqref{eq:perturbed_DI} (from Cor. \ref{cor:interconnected_cbf_solutions_to_perturbed_DI}). Also, notice that $\sigma$ is continuous, strictly decreasing on $\alpha$, and satisfies $\sigma(\alpha)\geq0$ and $\lim_{\alpha\to\infty}\sigma(\alpha)=0$. Thus, from Cor. \ref{cor:cbf_perturbation_of_limit_di}, as $\alpha\to\infty$, the perturbed DI \eqref{eq:perturbed_DI}, and thus the CBF-based dynamics \eqref{eq:interconnection_cbf}, tends to the \emph{limiting} DI \eqref{eq:limit_DI}; and, the larger $\alpha$ is picked, the closer is perturbed DI \eqref{eq:perturbed_DI}, and thus the CBF-based dynamics \eqref{eq:interconnection_cbf}, to the limiting DI \eqref{eq:limit_DI}. Finally, from Cor. \ref{cor:pds_same_solutions_with_limit_di}, we know that the limiting DI \eqref{eq:limit_DI} shares the same solutions with the PDS \eqref{eq:interconnection_pds}. Thus, CBF-based dynamics \eqref{eq:interconnection_cbf} are perturbed versions of PDS-based ones \eqref{eq:interconnection_pds}, and $\sigma(\alpha)$ is a bound on that perturbation, that vanishes as $\alpha\to\infty$. In the coming section, we employ all the above, to prove that trajectories of CBF-based dynamics \eqref{eq:interconnection_cbf} uniformly converge to trajectories of PDS-based dynamics \eqref{eq:interconnection_pds}.
\begin{remark}
{\rm
    \cite[Prop. 4.4]{allibhoy2023control} proves that the right-hand sides of \eqref{eq:interconnection_pds} and \eqref{eq:interconnection_cbf} coincide at $\alpha\to\infty$. This is generally not enough to infer convergence of solutions. In contrast, we provide bounds, depending on $\alpha$, on the perturbation that \eqref{eq:interconnection_cbf} is to \eqref{eq:interconnection_pds}. As it becomes evident in Section \ref{sec:convergence}, this enables proving convergence of solutions.
    } \hfill $\blacksquare$
\end{remark}

\subsection{Trajectories of CBF-based dynamics converge to trajectories of PDSs}\label{sec:convergence}
Here, we employ Cors. \ref{cor:interconnected_cbf_solutions_to_perturbed_DI}, \ref{cor:cbf_perturbation_of_limit_di} and \ref{cor:pds_same_solutions_with_limit_di}, to show that trajectories of CBF-based dynamics \eqref{eq:interconnection_cbf} uniformly converge to bounded trajectories of PDSs \eqref{eq:interconnection_pds}, \eqref{eq:interconnection_DI}, as $\alpha\to\infty$. This establishes that CBF-based dynamics are indeed approximations of PDSs. Specifically, we first show that any sequence of bounded solutions to the CBF-based dynamics, corresponding to an increasing sequence $\{\alpha_i\}_{i\in\mathbb{N}}$ of parameters $\alpha_i$ with $\lim_{i\to\infty}\alpha_i=\infty$, has a subsequence which converges to a solution of the PDS \eqref{eq:interconnection_pds}. Moreover, the whole sequence converges to a PDS solution \eqref{eq:interconnection_pds}, when this solution is unique.
\begin{theorem}\label{thm:solution_convergence}
    Consider vector fields $f:\real^m\times\real^n\to\real^n$ and $g:\real^m\times\real^n\to\real^m$, a compact set $\Zset\subset\real^m$ and a set $\Sset\subset\real^n$ defined by \eqref{eq:sset}. Let Assumptions \ref{assum:sset_and_h} and \ref{assum:f_lipschitz} hold. Consider an increasing sequence $\{\alpha_i\}_{i\in\mathbb{N}}$, such that $\alpha_i\geq\alpha_*$ and $\lim_{i\to\infty}\alpha_i=\infty$, where $\alpha_*$ is defined in Prop. \ref{prop:perturbation}. Given an initial condition $(z_0,x_0)\in\Zset\times\Sset$ and any $T\in[0,\infty)$, let the functions $\phi_i:[0,T]\to\real^m\times\real^n$, with $\phi_i(0)=(z_0,x_0)$, be the solutions to CBF-based dynamics \eqref{eq:interconnection_cbf} with $\alpha=\alpha_i$. Assume that $\phi_i(t)\in\Zset\times\Sset$ for all $t\in[0,T]$ and for all $i$. The following hold:
    \begin{enumerate}
        \item Let $g$ be continuous and locally bounded. Then, there is a subsequence $\{\phi_{i_k}\}_{k\in\mathbb{N}}$ of $\{\phi_i\}_{i\in\mathbb{N}}$ that converges uniformly, on any compact subinterval of $[0,T)$, to a solution $\phi$ of PDS \eqref{eq:interconnection_pds}, with $\mathrm{dom}\phi=[0,T]$ and $\phi(0)=(z_0,x_0)$.
        \item Further, let $g$ be locally Lipschitz. Then, the sequence $\{\phi_i\}_{i\in\mathbb{N}}$ converges uniformly, on any compact subinterval of $[0,T)$, to the unique solution $\phi$ of PDS \eqref{eq:interconnection_pds}, with $\mathrm{dom}\phi=[0,T]$ and $\phi(0)=(z_0,x_0)$.
    \end{enumerate}
\end{theorem}
\begin{remark}
{\rm
    Given Props. \ref{prop:perturbation}, \ref{prop:well-posed-limit-DI} and Cors. \ref{cor:interconnected_cbf_solutions_to_perturbed_DI},  \ref{cor:cbf_perturbation_of_limit_di}, \ref{cor:pds_same_solutions_with_limit_di}, the proof of Thm. \ref{thm:solution_convergence} follows steps similar to \cite[Section 4]{hauswirth2020anti}. 
    }\hfill $\blacksquare$
\end{remark}

Assuming that solutions of CBF-based dynamics \eqref{eq:interconnection_cbf} stay in $\Zset\times\Sset$, Thm. \ref{thm:solution_convergence} establishes their uniform convergence to solutions of PDSs \eqref{eq:interconnection_pds}. While it is guaranteed that $\xi$ stays in $\Sset$, this is not always the case for $\zeta$. Nevertheless, if PDS \eqref{eq:interconnection_pds} obeys a stability property, then, as the result below shows, solutions to \eqref{eq:interconnection_cbf} stay in $\Zset\times\Sset$, for sufficiently large $\alpha$, and thus converge to solutions of PDS \eqref{eq:interconnection_pds}.
\begin{theorem}\label{thm:stability_solution_convergence}
    Consider vector fields $f:\real^m\times\real^n\to\real^n$ and $g:\real^m\times\real^n\to\real^m$, a compact set $\Zset\subset\real^m$ and a set $\Sset\subset\real^n$ defined by \eqref{eq:sset}. Let Assumptions \ref{assum:sset_and_h} and \ref{assum:f_lipschitz} hold. Further, let $g$ be continuous and locally bounded. Given $\gamma>0$, consider compact sets $\Zset_0,\Zset'$, such that $\Zset_0\subseteq \Zset'$ and $\Zset'+\gamma\ball \subseteq \Zset$. Given any initial condition $(z_0,x_0)\in \Zset_0\times\Sset$ and a $T>0$, assume any solution $\phi:[0,T]\to\real^m\times\real^n$ to PDS \eqref{eq:interconnection_pds}, with $\phi(0)=(z_0,x_0)$, satisfies $\phi(t)\in \Zset'\times\Sset$ for all $t\in[0,T]$. Then, there exists $\alpha'\geq\alpha_*$, such that for any $\alpha\geq \alpha'$, the solution $\psi:[0,T]\to\real^m\times\real^n$ to the CBF-based dynamics \eqref{eq:interconnection_cbf}, with $\psi(0)=(z_0,x_0)$, satisfies $\psi(t)\in\Zset\times\Sset$ for all $t\in[0,T]$. Thus, for any increasing sequence $\{\alpha_i\}_{i\in\mathbb{N}}$, such that $\alpha_i\geq\alpha'$ and $\lim_{i\to\infty}\alpha_i=\infty$, and respective sequence of solutions $\{\psi_i\}_{i\in\mathbb{N}}$ to \eqref{eq:interconnection_cbf}, statements 1 and 2 of Thm. \ref{thm:solution_convergence} hold.
\end{theorem}
\begin{remark}
{\rm
    The stability assumption of Thm. \ref{thm:stability_solution_convergence} is often met in practice, as PDS-based controllers \cite{lorenzetti2022pi,hauswirth2020anti,deenen2021projection,fu2023novel,chu2023projection,hauswirth2021optimization} normally guarantee closed-loop stability of \eqref{eq:interconnection_pds}. Thus, if the PDS-based controller is replaced by a CBF-based one, to yield the CBF-based dynamics \eqref{eq:interconnection_cbf}, Thm. \ref{thm:stability_solution_convergence} ensures that trajectories of \eqref{eq:interconnection_cbf} converge to trajectories of \eqref{eq:interconnection_pds}; that is, \eqref{eq:interconnection_cbf} behaves similarly to \eqref{eq:interconnection_pds}.
    }
    \hfill $\blacksquare$
\end{remark} 
\begin{remark}\label{rem:sset and zset}
{\rm
    We study solutions that are constrained in the compact set $\Zset\times\Sset$. As $\Zset\times\Sset$ can be arbitrarily large, our results are semiglobal. Adopting the state-controller interconnection interpretation of \eqref{eq:interconnection_pds} and \eqref{eq:interconnection_cbf}, the set $\Sset$ describes constraints on controller states, whereas $\Zset$ is where state trajectories lie. The larger $\Zset$ is, the larger $\alpha_*$ is (Prop. \ref{prop:perturbation}). Thus, for smaller $\Zset$, one may use smaller values of $\alpha$, to approximate PDS-based dynamics with CBF-based ones.
    }\hfill $\blacksquare$
\end{remark}
\begin{remark}\label{rem:non-uniform}
{\rm
    As our results are semiglobal, one might be able to show non-uniform convergence in non-compact domains, through limit arguments over sequences of compact sets converging to $\real^n$.}\hfill $\blacksquare$
\end{remark}
\begin{remark}
{\rm
    Combining Theorem \ref{thm:stability_solution_convergence} and \cite[Prop. 6.34]{hybrid_book}, one can infer stability of CBF-based dynamics: if the origin for PDS-based dynamics \eqref{eq:interconnection_pds} is asymptotically stable, then it is practically stable for CBF-based dynamics \eqref{eq:interconnection_cbf}. Stability and contractivity results for CBF-based dynamics have also been reported in \cite{allibhoy2023anytime}.
    }\hfill $\blacksquare$
\end{remark}

\section{Numerical Examples}
\subsection{Feedback optimization}
\begin{figure*}[t!]
		\begin{subfigure}[t]{0.32\linewidth}
    \centering
    \includegraphics[scale=0.32]{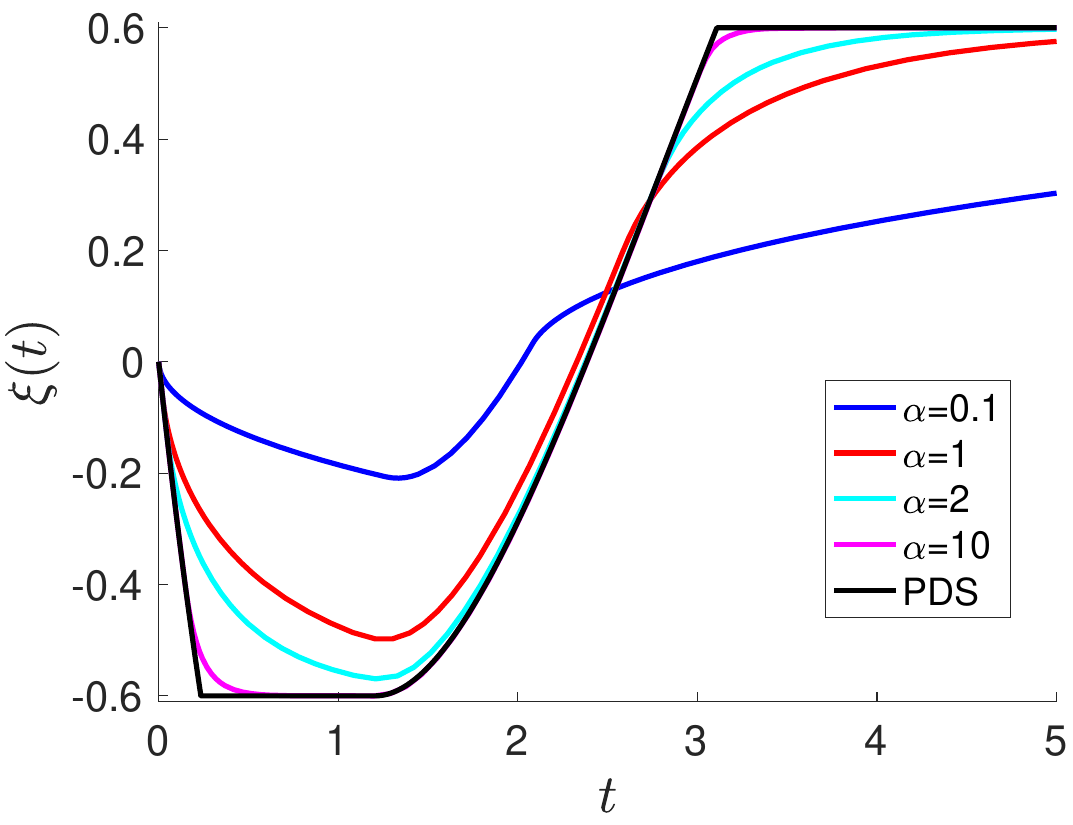}
    \caption{Evolution of the control input}
    \label{fig:feed_opti_input}
		\end{subfigure}~~~
		\begin{subfigure}[t]{0.32\linewidth}
    \centering
    \includegraphics[scale=0.32]{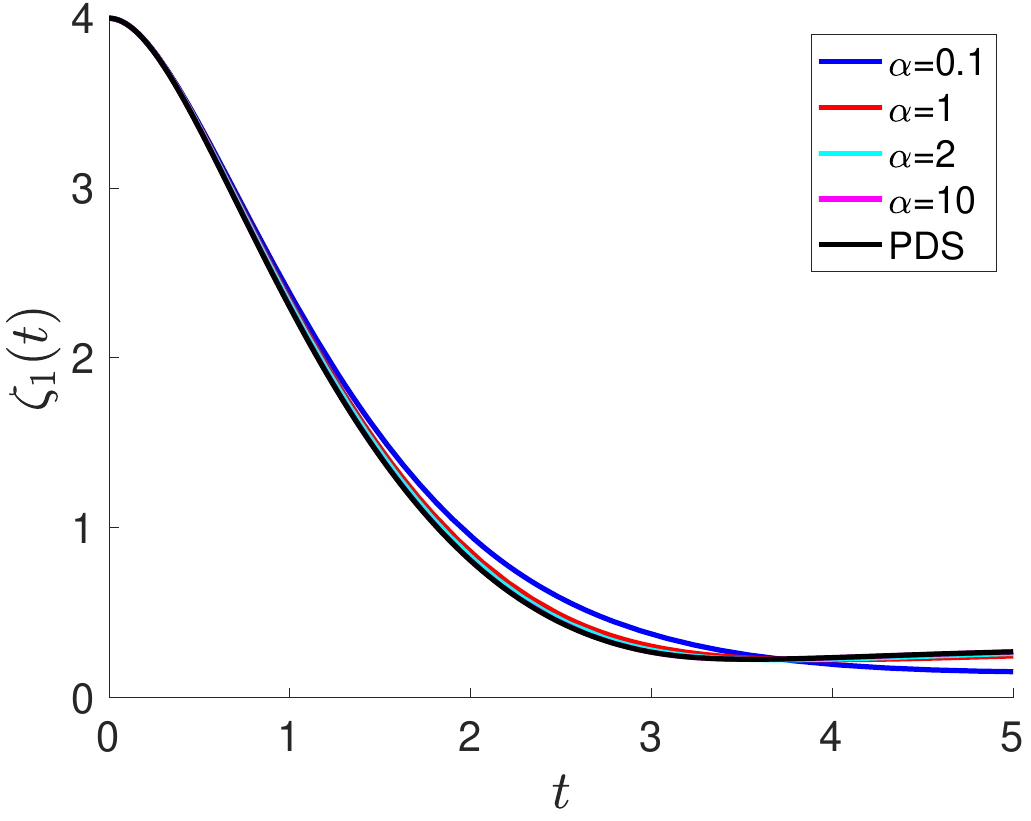}
    \caption{Evolution of state $\zeta_1$}
    \label{fig:feed_opti_state}
		\end{subfigure}~~
		\begin{subfigure}[t]{0.32\linewidth}
    \centering
    \includegraphics[scale=0.32]{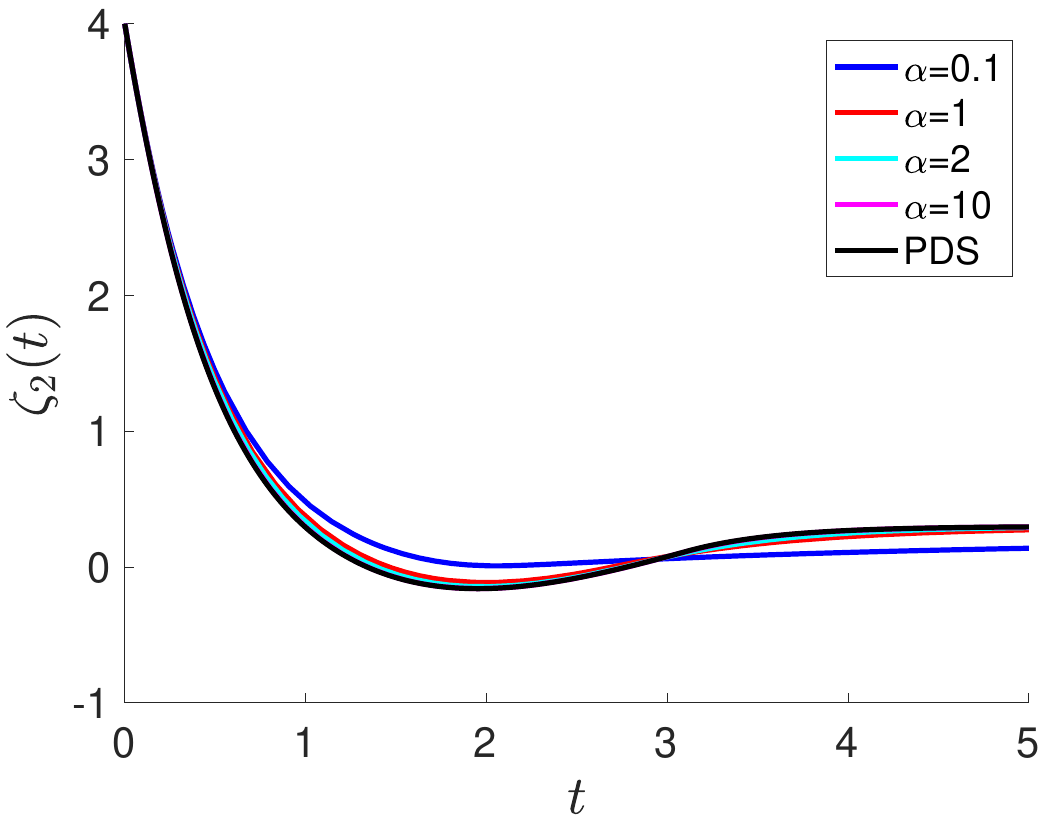}
    \caption{Evolution of state $\zeta_2$}
    \label{fig:feed_opti_state2}
		\end{subfigure}
		\caption{Evolution of the control input $\xi(t)$ and states $\zeta_1(t)$ and $\zeta_2(t)$ of system \eqref{eq:feedbck_opti_sys}, when the controller has PDS dynamics and CBF-based dynamics, for different values of $\alpha$.}
    \label{fig:states_input}
\end{figure*}
\begin{figure}[h!]
    \centering
    \includegraphics[scale=0.4]{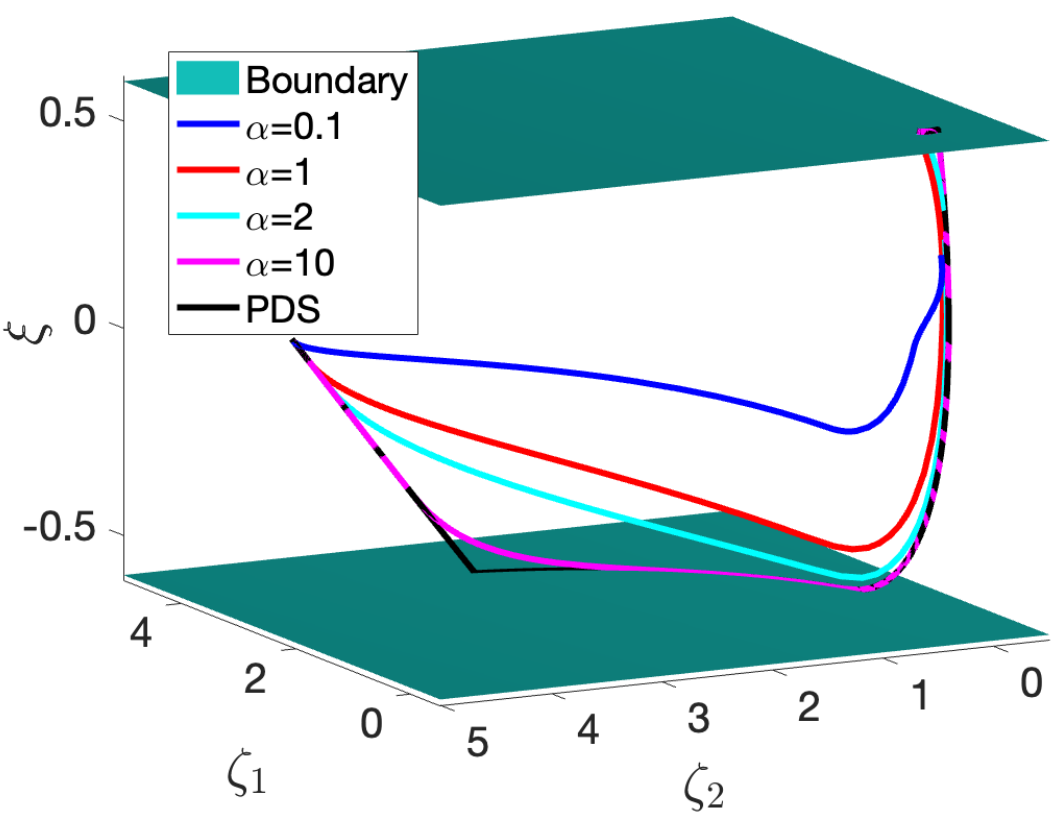}
    \caption{Trajectories $(\zeta(t),\xi(t))$ of system \eqref{eq:feedbck_opti_sys} coupled with dynamic controller, when the controller has PDS dynamics and CBF-based dynamics, for different values of $\alpha$.}
    \label{fig:feed_opti_3d}
\end{figure}
In \emph{feedback optimization}, a given system has to be steered via a dynamic controller to a steady state, which is specified as the solution to an optimization problem. This type of regulation finds applications in power grids and congestion control of communication networks, among others. See \cite{hauswirth2021optimization} and references therein for a thorough exposition. As an exemplary setup, we consider the plant
\begin{equation}\label{eq:feedbck_opti_sys}
    \dot{\zeta} = A\zeta + B\xi
\end{equation}
where $\xi$ is the state of a dynamic controller, to be designed. As commonly done in feedback-optimization scenarios (see \cite{hauswirth2021optimization}), we assume that $A$ is Hurwitz (i.e. the origin of the open-loop system is asymptotically stable). Thus, for every constant input $\xi_e$, there exists a unique steady-state equilibrium $\zeta_e = -A^{-1}B\xi_e$. We seek to drive the system to the solution of the following optimization problem:
\begin{equation}
\label{eq:feedback_opti_problem}
	\begin{aligned}
		\min_{\zeta,\xi} \quad \Phi(\zeta), \quad
		\mathrm{s.t.} \quad  \zeta = -A^{-1}B\xi
	\end{aligned}
\end{equation}
while respecting the input constraint $\xi(t)\in\Sset$ \emph{at all times}. The constraint in \eqref{eq:feedback_opti_problem} enforces exactly that the solution of the optimization problem has to comply with the fact that control system \eqref{eq:feedbck_opti_sys} behaves as the static input-output map $\zeta = -A^{-1}B\xi$, at the steady state.

Following \cite{hauswirth2020anti,hauswirth2021optimization}, under appropriate assumptions, the above problem can be solved by designing the dynamic controller as:
\begin{equation}\label{eq:feedback_opti_controller}
    \dot{\xi} \in \Pi_\Sset\Big(B^\top A^{-\top} \nabla \Phi(\zeta)\Big)
\end{equation}
which is a \emph{projected gradient flow}. In this example, we demonstrate how one may indeed approximate the PDS-based dynamics \eqref{eq:feedbck_opti_sys}-\eqref{eq:feedback_opti_controller}, by substituting the discontinuous controller \eqref{eq:feedback_opti_controller} with a CBF-based one, as in \eqref{eq:interconnection_cbf}. We consider the following data: 
\begin{align*}
    &A = \begin{bmatrix}
        -1 &1\\ 0 &-2
    \end{bmatrix}, \text{ } B = \begin{bmatrix}
        0\\ 1
    \end{bmatrix}, \text{ } \Phi(\zeta) = \Big(\zeta-\begin{bmatrix}
        1\\1
    \end{bmatrix}\Big)^2, \text{ }\Sset =  \{u\in\real:0.36-u^2\geq 0\}
\end{align*}
Observe that both Assumptions \ref{assum:sset_and_h} and \ref{assum:f_lipschitz} are satisfied by the data. In addition, as $A$ is Hurwitz and the controller $\xi(t)\in\Sset$, with $\Sset$ compact, it follows that trajectories of the PDS-based dynamics \eqref{eq:feedbck_opti_sys}-\eqref{eq:feedback_opti_controller} are bounded, and Thm. \ref{thm:stability_solution_convergence} applies.

We simulate trajectories of \eqref{eq:feedbck_opti_sys} coupled with a dynamic controller, which has either the PDS dynamics \eqref{eq:feedback_opti_controller} or CBF-based approximations thereof, for different values of $\alpha$. Figures \ref{fig:states_input} and \ref{fig:feed_opti_3d} depict the results. As expected from Thms. \ref{thm:solution_convergence} and \ref{thm:stability_solution_convergence}, we observe that, as $\alpha$ becomes larger, trajectories of the CBF-based dynamics approximate more closely trajectories of PDS-based ones. Moreover, as expected from \cite[Thms. 5.6 and 5.7]{allibhoy2023anytime}, since $\Sset$ and \eqref{eq:feedback_opti_problem} are convex, both PDS-based dynamics and CBF-based ones converge to the global minimizer $(\zeta_\star, \xi_\star) = (0.3,0.3,0.6)$.

\subsection{Control of a synchronverter}
\begin{figure*}[t!]
	\begin{subfigure}[t]{0.33\linewidth}
        \centering
        \includegraphics[scale=0.31]{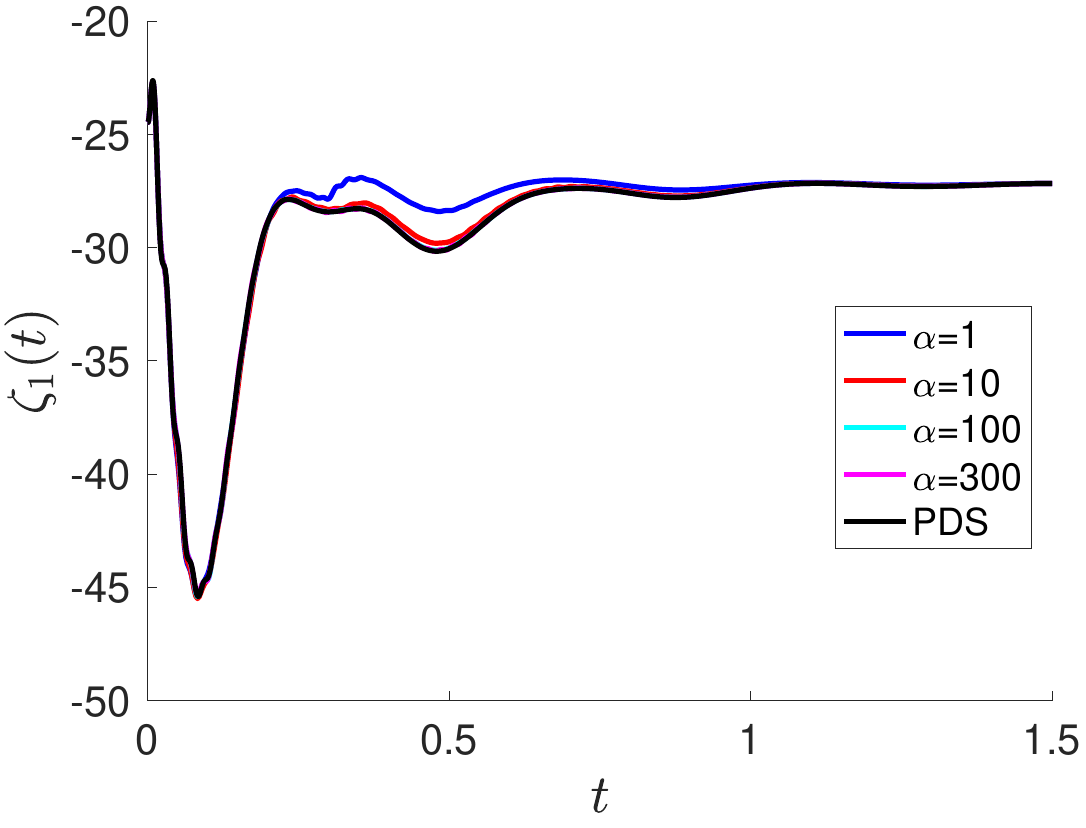}
        \caption{Evolution of state $\zeta_1 = i_d$}
        \label{fig:zeta}
	\end{subfigure}~~~
	\begin{subfigure}[t]{0.33\linewidth}
        \centering
        \includegraphics[scale=0.31]{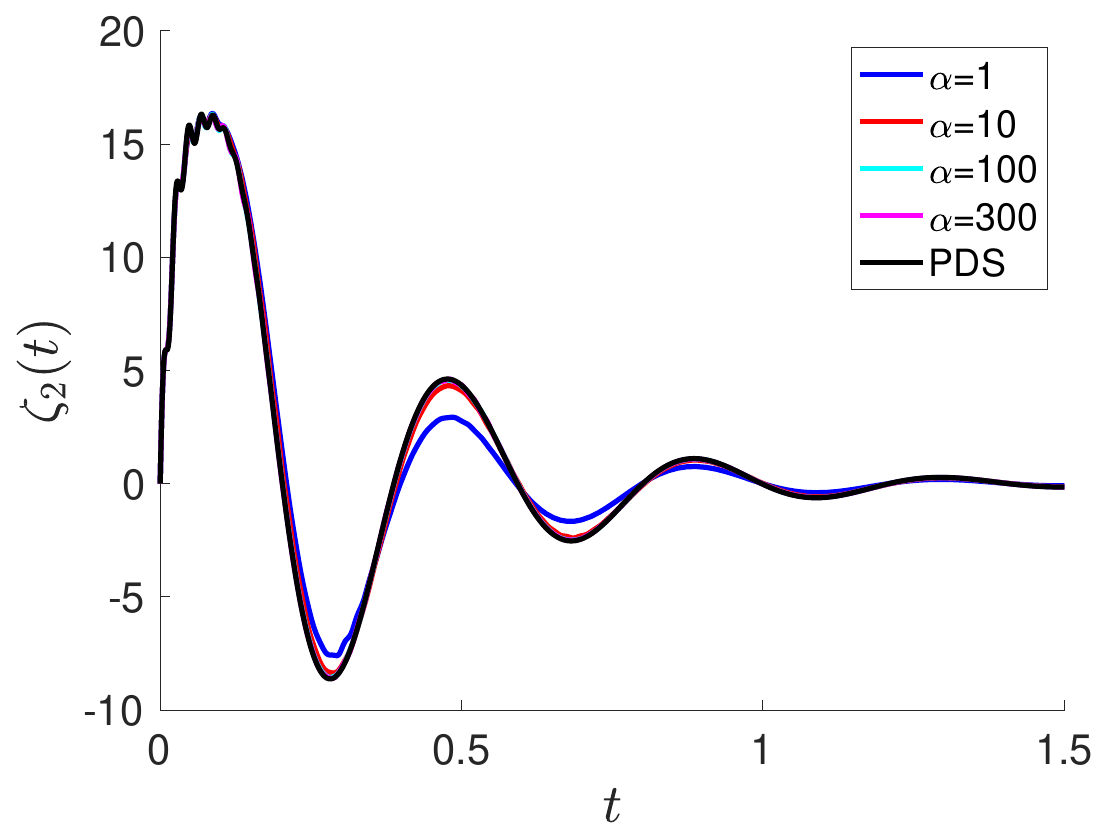}
        \caption{Evolution of state $\zeta_2=i_q$}
        \label{fig:zeta2}
	\end{subfigure}~~~
	\begin{subfigure}[t]{0.33\linewidth}
        \centering
        \includegraphics[scale=0.31]{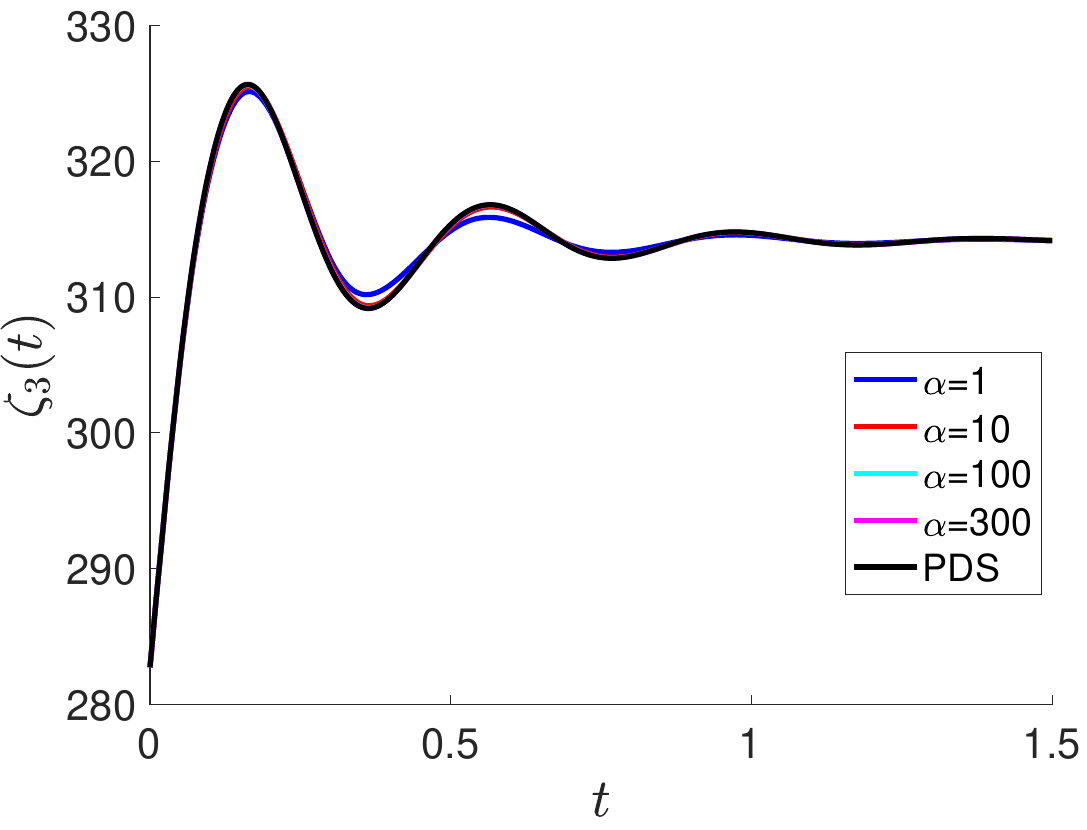}
        \caption{Evolution of state $\zeta_3=\omega$}
        \label{fig:zeta3}
	\end{subfigure}\\
    \begin{subfigure}[t]{0.33\linewidth}
        \centering
        \includegraphics[scale=0.31]{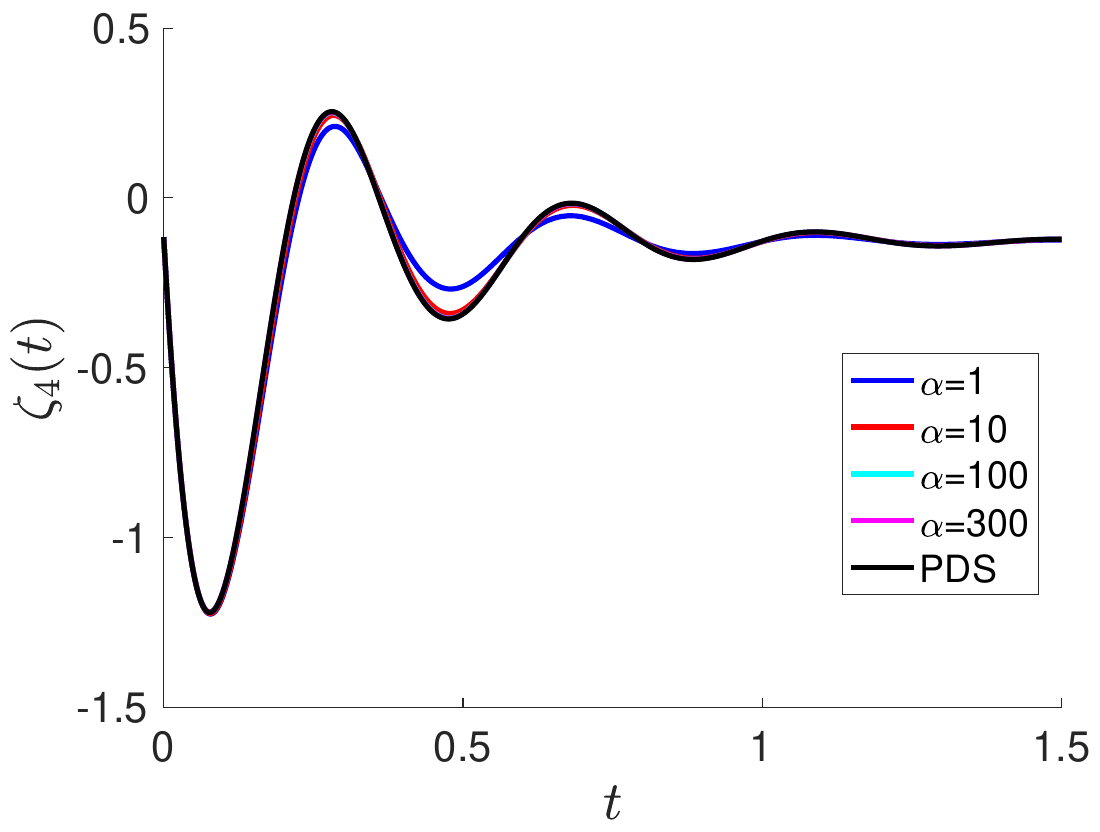}
        \caption{Evolution of state $\zeta_4=\delta$}
        \label{fig:zeta4}
	\end{subfigure}~~~
    \begin{subfigure}[t]{0.33\linewidth}
        \centering
        \includegraphics[scale=0.31]{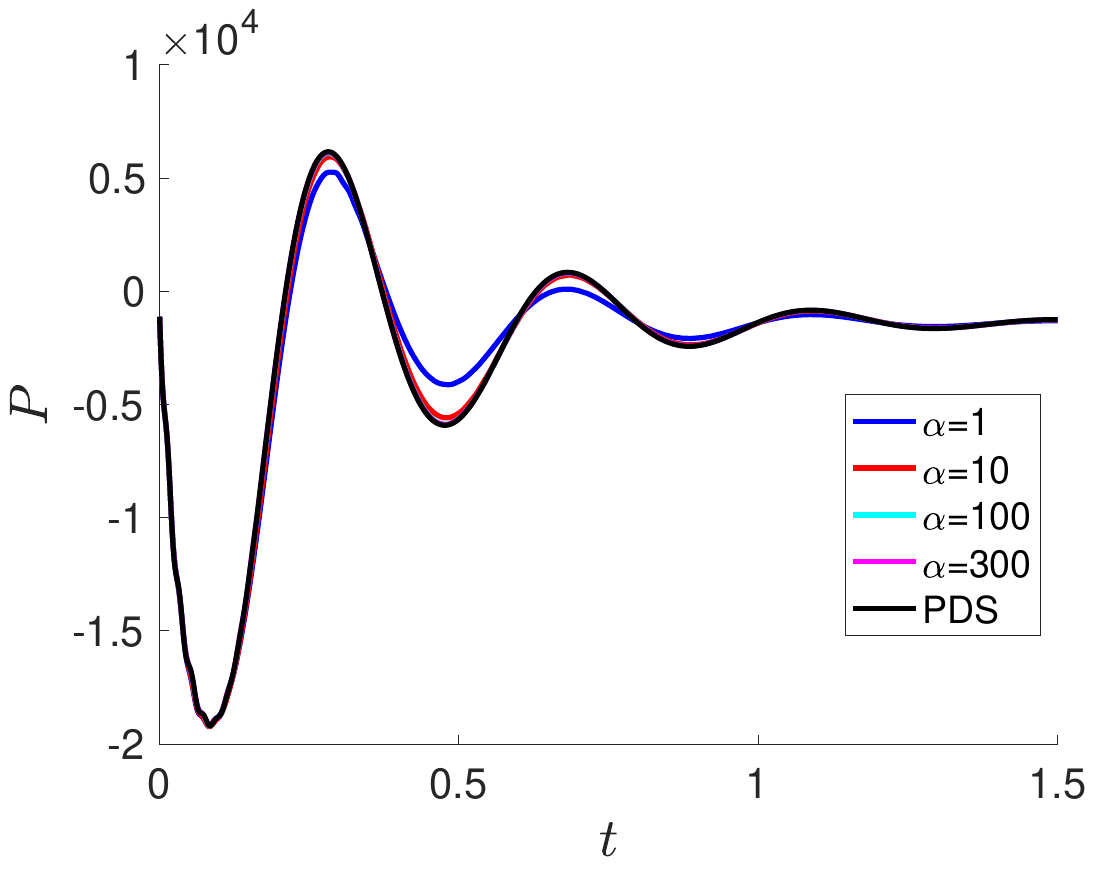}
        \caption{Evolution of output $y_1 = P$}
        \label{fig:P}
	\end{subfigure}~~~
	\begin{subfigure}[t]{0.33\linewidth}
        \centering
        \includegraphics[scale=0.31]{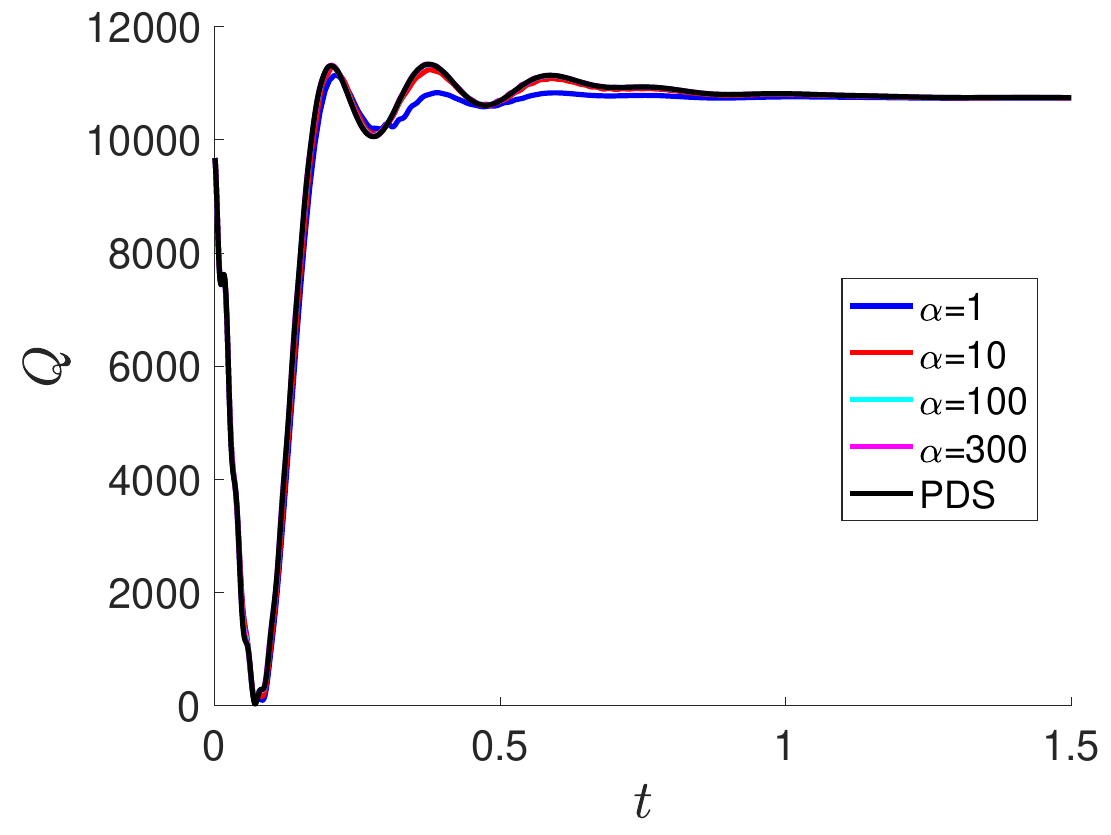}
        \caption{Evolution of output $y_2=Q$}
        \label{fig:Q}
	\end{subfigure}
    \caption{Evolution of the state $\zeta(t)$ and outputs $P,Q$, for controllers with PDS-based dynamics and CBF-based ones, for different $\alpha$.}
    \label{fig:lorenzetti}
\end{figure*}
\begin{figure}[h!]
    \centering
    \includegraphics[scale=0.3]{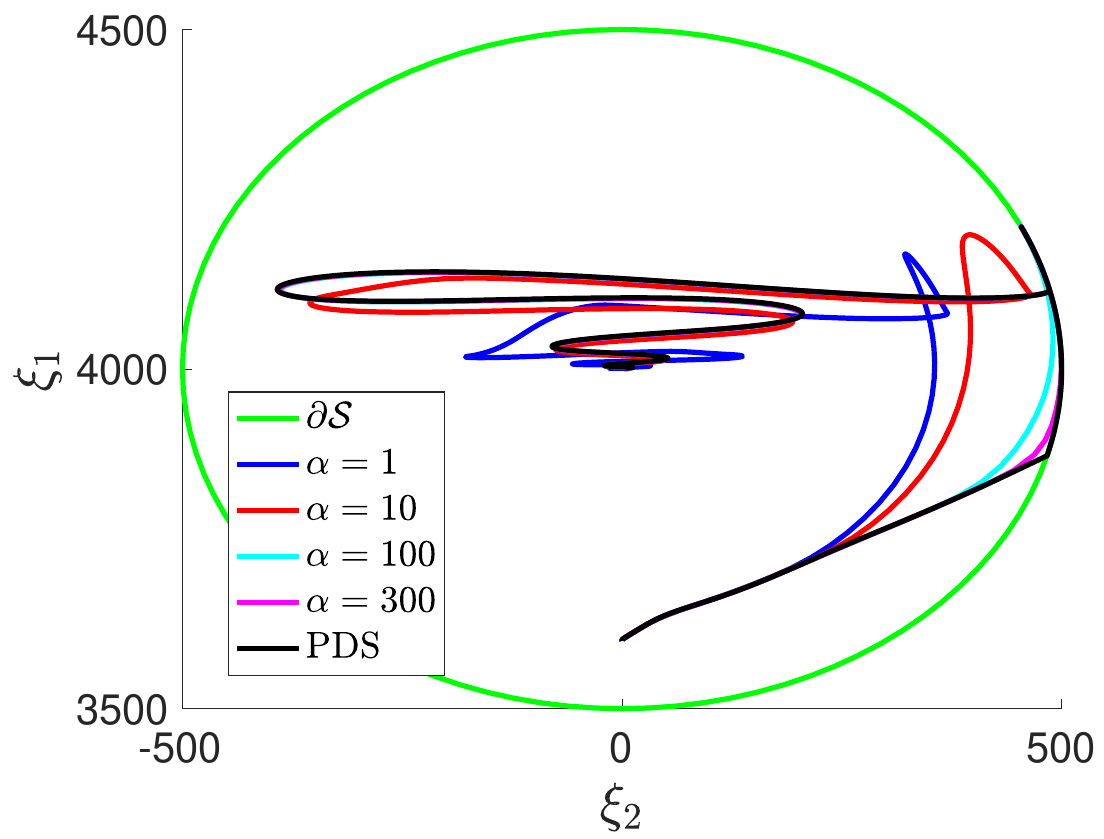}
    \caption{Trajectories $\xi(t)$ of the dynamic controller with PDS-based dynamics and CBF-based ones, for different $\alpha$.}
    \label{fig:xi_plane}
\end{figure}
Synchronverters are inverters that behave like synchronous generators (see \cite{lorenzetti2022pi}). The synchronverter's state is $\zeta = \begin{bmatrix}
    i_d &i_q &\omega &\delta
\end{bmatrix}^\top$, where $i_d$ and $i_q$ are the stator currents' $d$ and $q$ components, $\omega$ is the (virtual) rotor angular velocity, and $\delta$ is the power angle. The control input is $u=\begin{bmatrix}
    T_m &i_f
\end{bmatrix}^\top$, where $T_m$ is the (virtual) prime mover torque and $i_f$ is the (virtual) field current. The $\zeta$-dynamics is
\begin{equation} \label{eq:lorenzetti_state}
    \dot{\zeta} = H^{-1}A(\zeta,u)\zeta + H^{-1}v(\zeta,u), \quad y = G(\zeta)
\end{equation}
where $y=\begin{bmatrix}
    P &Q
\end{bmatrix}^\top$ is the active power $P$ and reactive power $Q$,
\begin{align*}
    &H=\begin{bmatrix}
            L &0 &0 &0\\
            0 &L &0 &0\\
            0 &0 &J &0\\
            0 &0 &0 &1\end{bmatrix}, \text{ }A(\zeta,u) = \begin{bmatrix}
        -R &\omega L &0 &0\\
        -\omega L &-R &-mi_f &0\\
        0 &-mi_f &-D_p &0\\
        0 &0 &1 &0
    \end{bmatrix}, \\ &v(\zeta,u) = \begin{bmatrix}
        V\sin\delta\\
        V\cos\delta\\
        T_m+D_p\omega_n\\
        -\omega_g 
    \end{bmatrix}, \text{ }G(\zeta) = -V\begin{bmatrix}
            \cos\delta &\sin\delta\\
            -\sin\delta &\cos\delta
        \end{bmatrix}\begin{bmatrix}
            i_q\\ i_d
        \end{bmatrix}
\end{align*}
and $V=230\sqrt{3}\unit{V}$, $L=56.75\unit{mH}$, $J=0.2\unit{m^2/rad}$, $R = 1.875 \unit{\Omega}$, $D_p = 3\unit{Nm/(rad/s)}$, $m=3.5\unit{H}$, $\omega_g=\omega_n=100\pi \unit{rad/s}$. 

The system admits a steady-state input-output map (see \cite{lorenzetti2022pi}). The control problem is to regulate the output $y$ to a reference $r$. In \cite{lorenzetti2022pi}, a PDS-like dynamic controller is proposed, for antiwindup control:
\begin{equation} \label{eq:lorenzetti_control}
    \dot{\xi} = \Pi_\Sset(\xi, r-y), \quad u = \mathcal{N}(\xi)
\end{equation}
where the constraint set $\Sset$ and the map $\mathcal{N}$ are to be designed.

We select $r = \begin{bmatrix}
    -1385\unit{W} &10738\unit{Var}
\end{bmatrix}$. A steady-state input that corresponds to this reference is $u_s = \begin{bmatrix}
    0 &0.8
\end{bmatrix}^\top$. The corresponding equilibrium is locally asymptotically stable, which enables Thm. \ref{thm:stability_solution_convergence}, as long as the system operates close to the equilibrium (which is a typical assumption). Following \cite{lorenzetti2022pi}, we choose $\mathcal{N}(\xi) = \begin{bmatrix}
    \tfrac{1}{50} &0\\ 0 &\tfrac{1}{5000}
\end{bmatrix}\xi$. We also select $\Sset = \Big\{x\in\real^2: \text{ }500^2 - x_1^2 - (x_2-4000)^2\geq 0 \Big\}$. Assumptions \ref{assum:sset_and_h} and \ref{assum:f_lipschitz} are satisfied. We simulate system \eqref{eq:lorenzetti_state} coupled with a controller with either PDS-based dynamics \eqref{eq:lorenzetti_control} or CBF-based ones, for different values of $\alpha$. The results are depicted in Figures \ref{fig:lorenzetti} and \ref{fig:xi_plane}. We observe that, as $\alpha$ becomes larger, trajectories of the CBF-based dynamics approximate more closely trajectories of PDS-based ones, confirming Thms. \ref{thm:solution_convergence} and \ref{thm:stability_solution_convergence}. Especially, cyan ($\alpha=100$) and magenta ($\alpha=300$) trajectories are almost indistinguishable from PDS ones. Moreover, as expected, the input constraints are respected by all controllers. Finally, the CBF-based controller drives the output to the reference, for any chosen $\alpha$.

\section{Conclusion}
We have proven that trajectories of CBF-based dynamics uniformly converge to those of PDS-based ones, as the parameter $\alpha$ of the CBF constraint approaches $\infty$. We have also proven that CBF-based dynamics are perturbations of PDSs, and quantified the perturbation. One can use our results to implement discontinuous PDS-based controllers in a continuous manner, employing CBFs. We have demonstrated this on a synchronverter control example. Our results also enable novel numerical simulation schemes for PDSs, through CBF-based dynamics, overcoming disadvantages of existing ones, such as projections to possibly non-convex sets. Finally, this bridge between PDSs and CBFs might yield other results in the future, such as insights on stability.
\section{Acknowledgements}
The authors would like to thank Dr. Pietro Lorenzetti for his invaluable help on the implementation of the synchronverter numerical example.
\section{Technical results and proofs} \label{sec:proofs}
\subsection{Proof of Prop. \ref{prop:perturbation}}\label{sec:proof_theorem_perturbation_open_loop}
To prove Prop. \ref{prop:perturbation}, we use Lemmas \ref{lem:neighbourhood}, \ref{lem:nonzero_grad} and \ref{lem:normalized_grad_lipschitz} below, which have already appeared in \cite{delimpaltadakis2023relationship}, in a simpler form. Here they are extended, to account for interconnections \eqref{eq:interconnection_DI} and \eqref{eq:interconnection_cbf}, rather than autonomous setups \eqref{eq:pds_di} and \eqref{eq:sys_cbf}, which was the context of \cite{delimpaltadakis2023relationship}.\footnote{Here, $f:\real^m\times\real^n\to\real^n$, that is $f$ has both $\xi$ and $\zeta$ as arguments.} Moreover, in \cite{delimpaltadakis2023relationship} the proofs of Lemmas \ref{lem:neighbourhood} and \ref{lem:nonzero_grad} were only sketched. Here, we provide the full proofs of Lemmas \ref{lem:neighbourhood} and \ref{lem:nonzero_grad} and recall Lemma \ref{lem:normalized_grad_lipschitz}. The following results are stated under the assumptions of Prop. \ref{prop:perturbation}. Given $z\in\Zset$, denote $U_{\mathrm{cbf},\alpha}(z):=\{x\in\Sset\mid \lie_f h(z,x) + \alpha h(x)\leq 0\}$.  
\begin{lemma}\label{lem:neighbourhood}
    Given $\alpha>0$ and $z\in\Zset$, consider $x\in U_{\mathrm{cbf},\alpha}(z)$ and $y\in\proj_{\partial\Sset}(x)$. It holds that $\|x-y\|\leq\gamma(\frac{1}{\alpha}|\lie_f h(z,x)|)$.
\end{lemma}
\begin{proof}[Proof of Lemma \ref{lem:neighbourhood}]
    Assumption \ref{assum:sset_and_h} item \ref{assum:Kinf} implies that: 
    \begin{align*}
        \|x-y\|=d(x,\partial\Sset)\leq\gamma(h(x))&\leq\gamma(-\frac{1}{\alpha}\lie_f h(z,x))= \gamma(\frac{1}{\alpha}|\lie_f h(z,x)|)
    \end{align*}
    where we used $0\leq h(x) \leq -\frac{1}{\alpha}\lie_f h(z,x)$.
\end{proof}
\begin{lemma}\label{lem:nonzero_grad}
For any $\alpha\geq \alpha_*$ and $z\in\Zset$, for all $x\in U_{\mathrm{cbf},\alpha}(z)$, we have $0<\epsilon\leq\|\nabla h(x)\|\leq M_3$.
\end{lemma}
\begin{proof}[Proof of Lemma \ref{lem:nonzero_grad}] Given $\alpha>0$, consider any $x\in U_{\mathrm{cbf},\alpha}(z)$. From Assumption \ref{assum:h_C11}, we have the following, for any $y\in\mathrm{proj}_{\partial\Sset}(x)$:
    \begin{align*}
        \|\nabla h(x) - \nabla h(y) \|\leq L_{\nabla h}\|x-y\|\implies 
        \|\nabla h(x)\|\geq \|\nabla h(y)\|-L_{\nabla h}\|x-y\|\implies&\\
        \|\nabla h(x)\|\geq\|\nabla h(y)\| - L_{\nabla h}\gamma(|\frac{1}{\alpha}\lie_f h(z,x)|)\quad&
    \end{align*}
    where in the last step we have used Lemma~\ref{lem:neighbourhood}. Thus, for any $\alpha\geq \alpha_*$, and for any $x\in U_{\mathrm{cbf},\alpha}(z)$, we get:
    \begin{align*}
        \|\nabla h(x)\|&\geq M_1 - L_{\nabla h}\gamma(\frac{1}{\alpha}\max_{(z,x)\in\Zset\times\Sset}|\lie_f h(z,x)|) \\&\geq M_1 - L_{\nabla h}\gamma(\frac{1}{\alpha_*}\max_{(z,x)\in\Zset\times\Sset}|\lie_f h(z,x)|) = \epsilon >0
    \end{align*}
    Similarly, for any $\alpha\geq \alpha_*$ and for any $x\in U_{\mathrm{cbf},\alpha}(z)$, $\|\nabla h(x) - \nabla h(y) \|\leq L_{\nabla h}\|x-y\|$ implies:
    \begin{align*}
        \|\nabla h(x)\|\leq \|\nabla h(y)\|+L_{\nabla h}\|x-y\|&\leq M_2 + L_{\nabla h}\gamma(\frac{1}{\alpha}|\lie_f h(z,x)|)\\
        &\leq M_2 + \gamma(\frac{1}{\alpha_*}|\lie_f h(z,x)|)\\
        &\leq M_3
    \end{align*}
\end{proof}
\begin{lemma}\label{lem:normalized_grad_lipschitz}
    Given any $\alpha\geq \alpha_*$ and $z\in\Zset$, it holds that, for any $x\in U_{\mathrm{cbf},\alpha}(z)$ and $y\in\mathrm{proj}_{\partial\Sset}(x)$
    \begin{equation*}
        \Big\|\frac{P^{-1}\nabla h(x)}{\|\nabla h(x)\|^2_{P^{-1}}} - \frac{P^{-1}\nabla h(y)}{\|\nabla h(y)\|^2_{P^{-1}}}\Big\|\leq L_1\|x-y\|
    \end{equation*}
    where $L_1 := \frac{\overline{\lambda}(P)}{\underline{\lambda}(P)\epsilon^2}L_{\nabla h}\biggl[1 + \frac{M_2\overline{\lambda}(P)\Big(M_2+M_3\Big)}{\underline{\lambda}(P)M_1^2}\bigg]$.
\end{lemma}
\begin{proof}[Proof of Lemma \ref{lem:normalized_grad_lipschitz}] See \cite{delimpaltadakis2023relationship}.
\end{proof}

We are ready to prove Prop. \ref{prop:perturbation}.
\begin{proof}[Proof of Prop. \ref{prop:perturbation}] Observe that both $\sigma(\alpha)$ and $\delta$ are well-defined, as they are maxima of continuous locally bounded functions over compact sets. For $x \in \Sset$, we distinguish the following cases:

\paragraph{Case 1} $\lie_f h(z,x) + \alpha h(x)>0$. Here we have $f_{\mathrm{cbf},\alpha}(z,x)=f(z,x)$, and the result holds trivially.

\paragraph{Case 2} $\lie_f h(z,x) + \alpha h(x)\leq 0$. In this case, $x\in U_{\mathrm{cbf},\alpha}(z)$ and we can write
\begin{equation}
    f_{\mathrm{cbf},\alpha}(z,x) = f(z,x) - \frac{\lie_f h(z,x) + \alpha h(x)}{\|\nabla h(x)\|^2_{P^{-1}}}P^{-1}\nabla h(x)
\end{equation}
which is well-defined, as $\nabla h(x)\neq 0$, from Lemma \ref{lem:nonzero_grad}.

Consider any $y\in\mathrm{proj}_{\partial\Sset}(x)$. Observe that $\eta := \Big(\lie_f h(z,x) + \alpha h(x)\Big)\frac{\nabla h(y)}{\|\nabla h(y)\|^2_{P^{-1}}}\in N_{\Sset}(y)$, since $\frac{\lie_f h(z,x) + \alpha h(x)}{\|\nabla h(y)\|^2_{P^{-1}}}\leq 0$. Also, $\nabla h(y) \neq 0$, due to Assumption \ref{assum:sset_and_h} item \ref{assum:h_C11}. We have the following:
{\allowdisplaybreaks
\begin{align*}
     \Big\|f(z,y) - P^{-1}\eta - f(z,x) + \frac{\lie_f h(z,x) + \alpha h(x)}{\|\nabla h(x)\|^2_{P^{-1}}}P^{-1}\nabla h(x)\Big\|\leq&\\
    \|f(z,y)-f(z,x)\| + |\lie_f h(z,x) + \alpha h(x)|\Big\|\frac{P^{-1}\nabla h(y)}{\|\nabla h(y)\|^2_{P^{-1}}}-\frac{P^{-1}\nabla h(x)}{\|\nabla h(x)\|^2_{P^{-1}}}\Big\|\leq&\\
    L_f\|x-y\|+|\lie_f h(z,x) + \alpha h(x)|L_1\|x-y\|\leq&\\
    L_f\|x-y\|+|\lie_f h(z,x)|L_1\|x-y\|\quad\leq&\\
    \underbrace{\Big(L_f+L_1|\lie_f h(z,x)|\Big)\gamma(\frac{1}{\alpha}|\lie_f h(z,x)|)}_{\sigma_1(\alpha,z,x)}\quad
\end{align*}}
where in the second inequality we used Assumption \ref{assum:f_lipschitz}, in the third one we used Lemma \ref{lem:normalized_grad_lipschitz}, in the fourth one we used $|\lie_f h(z,x)+\alpha h(x)|\leq|\lie_f h(z,x)|$, and in the fifth one we used Lemma \ref{lem:neighbourhood}. 

Regarding $P^{-1}\eta$, notice that
\allowdisplaybreaks{\begin{align*}
    \|P^{-1}\eta\| &\leq \frac{1}{\underline{\lambda}(P)}\|\eta\|\\
    &\leq\frac{1}{\underline{\lambda}(P)}|\lie_f h(z,x) + \alpha h(x)| \Big\|\frac{\nabla h(y)}{\|\nabla h(y)\|^2_{P^{-1}}}\Big\|\\
    &\leq \frac{1}{\underline{\lambda}(P)}|\lie_f h(z,x)|\frac{\|\nabla h(y)\|}{\|\nabla h(y)\|^2_{P^{-1}}}\\
    &\leq \frac{1}{\underline{\lambda}(P)}\|f(z,x)\|\|\nabla h(x)\|\frac{\|\nabla h(y)\|}{\|\nabla h(y)\|^2_{P^{-1}}}\\
    &\leq \frac{\|f(z,x)\|}{\underline{\lambda}(P)}\Big(\|\nabla h(y)\|+L_{\nabla h}\|x-y\|\Big)\frac{\|\nabla h(y)\|}{\|\nabla h(y)\|^2_{P^{-1}}}\\
    &\leq \frac{1}{\underline{\lambda}(P)}\Big(\frac{1}{\underline{\lambda}(P^{-1})}+\frac{L_{\nabla h}\|x-y\|}{\underline{\lambda}(P^{-1})\|\nabla h(y)\|}\Big)\|f(z,x)\|\\
    &\leq \Big(1+\frac{L_{\nabla h}\gamma(\frac{1}{\alpha}|\lie_f h(z,x)|)}{M_1}\Big)\frac{\overline{\lambda}(P)}{\underline{\lambda}(P)}\|f(z,x)\|\\
    &\leq\delta
\end{align*}}
where in the fifth inequality we used $\|\nabla h(x)\|\leq \|\nabla h(y)\|+L_{\nabla h}\|x-y\|$, as proven in the proof of Lemma \ref{lem:nonzero_grad}, and in the last one we used Lemma \ref{lem:neighbourhood}.

Finally, from the above we get:
\allowdisplaybreaks{\begin{align*}
    f_{\mathrm{cbf},\alpha}(x) =&\\ 
    f(z,x) - \frac{\lie_f h(z,x) + \alpha h(x)}{\|\nabla h(x)\|^2_{P^{-1}}}P^{-1}\nabla h(x) \in&\\
    f(z,y) - P^{-1}\eta+\sigma_1(\alpha,z,x)\ball\subseteq&\\
    f(z,y) - P^{-1}N_{\Sset}(y)\cap\delta\ball+\sigma_1(\alpha,z,x)\ball\subseteq&\\
    f\Big((x+\gamma(\frac{1}{\alpha}|\lie_f h(z,x)|)\ball)\cap\Sset\Big) - P^{-1}N_{\Sset}\Big((x+\gamma(\frac{1}{\alpha}|\lie_f h(z,x)|)\ball)\cap\Sset\Big)\cap\delta\ball+\sigma_1(\alpha,z,x)\ball\subseteq&\\
    f\Big((x+\sigma(\alpha)\ball)\cap\Sset\Big) - P^{-1}N_{\Sset}\Big((x+\sigma(\alpha)\ball)\cap\Sset\Big)\cap\delta\ball+\sigma(\alpha)\ball=&\\
    K_{\sigma(\alpha)}\bigg(\normaldi^P_\Sset\Big(x,f(z,x),\delta\Big)\bigg)\quad& 
\end{align*}}
\end{proof}

\subsection{Proofs of Thms. \ref{thm:solution_convergence} and \ref{thm:stability_solution_convergence}}
To prove Thms. \ref{thm:solution_convergence} and \ref{thm:stability_solution_convergence}, we employ several results from \cite{hybrid_book}. The first step to proving Thms. \ref{thm:solution_convergence} and \ref{thm:stability_solution_convergence} is to establish \emph{well-posedness}\footnote{Well-posedness of a DI implies that solutions of its $\sigma$-perturbations converge to the DI's solutions, as $\sigma\to 0$. See \cite[Def. 6.29]{hybrid_book}.}, in the sense of \cite[Def. 6.29]{hybrid_book}, of limiting DI \eqref{eq:limit_DI}.
\begin{proposition}\label{prop:well-posed-limit-DI}
    Consider DI \eqref{eq:limit_DI}. Let Assumptions \ref{assum:sset_and_h} and \ref{assum:f_lipschitz} hold. Let $g$ be continuous and locally bounded. Then, DI \eqref{eq:limit_DI} is well-posed.
\end{proposition}
\begin{proof}
    First, $\real^m\times\Sset$ is closed. Second, $g$ is locally bounded and outer semicontinuous. Being a singleton, it is also a convex set-valued map. Further, $\normaldi^P_\Sset\Big(x,f(z,x),\delta\Big)$ is: (a) locally bounded, as $f$ is locally Lipschitz and the normal cone is truncated by the bounded quantity $\delta$; (b) convex, as a translation of the intersection of two convex sets (the normal cone and a ball); (c) outer semicontinuous, since its graph is closed (outer semicontinuity of $N_\Sset(x)$ follows from \cite[Prop. 6.6]{rockafellar2009variational}). Thus, by \cite[Thm. 6.30]{hybrid_book}, the result follows.
\end{proof}

We are ready to prove Thms. \ref{thm:solution_convergence} and \ref{thm:stability_solution_convergence}.
\begin{proof}[Proof of Thm. \ref{thm:solution_convergence}]
    We follow steps similar to \cite[Section 4]{hauswirth2020anti}. We prove graphical convergence\footnote{For the definition of graphical convergence, see \cite[Def. 5.18]{hybrid_book}.} of $\{\phi_{i_k}\}_{k\in\mathbb{N}}$ or $\{\phi_i\}_{i\in\mathbb{N}}$ to $\phi$. Thus, we regularly treat $\phi_{i_k}$, $\phi_i$ and $\phi$ as sets, through their graphs, below. From \cite[Lemma 5.28]{hybrid_book}, graphical convergence implies uniform convergence on compact subintervals of $[0,T)$. Further, since $\sigma(\cdot)$ is decreasing and $\lim_{\alpha\to\infty}\sigma(\alpha)=0$, the sequence $\sigma(\alpha_i)$ is decreasing and $\lim_{i\to\infty}\sigma(\alpha_i)=0$.

    \textbf{Proof of Statement 1.} From \cite[Thm. 5.7]{hybrid_book}, $\{\phi_i\}_{i\in\mathbb{N}}$ has a graphically convergent subsequence, say $\{\phi_{i_k}\}_{k\in\mathbb{N}}$. Note that $\phi_i$ are also solutions to perturbed DI \eqref{eq:perturbed_DI} for $\alpha=\alpha_i$ (Cor. \ref{cor:interconnected_cbf_solutions_to_perturbed_DI}), i.e., $\phi_i$, and thus $\phi_{i_k}$, are solutions to $\sigma(\alpha_i)$-perturbations, and $\sigma(\alpha_{i_k})$-perturbations, respectively, of the limiting DI \eqref{eq:limit_DI} (Cor. \ref{cor:cbf_perturbation_of_limit_di}). Thus, since the  limiting DI \eqref{eq:limit_DI} is well-posed (Prop. \ref{prop:well-posed-limit-DI}), and since $\{\phi_{i_k}\}_{k\in\mathbb{N}}$ is convergent, then $\{\phi_{i_k}\}_{k\in\mathbb{N}}$ graphically converges to a solution $\phi$ of the limiting DI \eqref{eq:limit_DI}, from \cite[Def. 6.29]{hybrid_book}. Finally, from Cor. \ref{cor:pds_same_solutions_with_limit_di} item 1, $\phi$ is a solution to PDS \eqref{eq:interconnection_pds}.

    \textbf{Proof of Statement 2.} Similar to the proof of \cite[Cor. 4.6]{hauswirth2020anti}.
\end{proof}

\begin{proof}[Proof of Thm. \ref{thm:stability_solution_convergence}]
    Let $\gamma'>0$ with $\gamma'<\gamma$. Limiting DI \eqref{eq:limit_DI} is well-posed (Prop. \ref{prop:well-posed-limit-DI}) and has the same solutions as PDS \eqref{eq:interconnection_pds} (Cor. \ref{cor:pds_same_solutions_with_limit_di}). Perturbed DI \eqref{eq:perturbed_DI} is a $\sigma(\alpha)$-perturbation to limiting DI \eqref{eq:limit_DI} (Cor. \ref{cor:cbf_perturbation_of_limit_di}). Then, from \cite[Prop. 6.34]{hybrid_book}, $\exists \alpha'\geq\alpha_*$, such that for all $\alpha\geq\alpha'$ every solution $\chi$ to perturbed DI \eqref{eq:perturbed_DI} is $\gamma'$-close to a solution of \eqref{eq:limit_DI}, implying $\chi(t)\in(\Zset'+\gamma'\ball)\times\Sset$, $\forall t\in[0,T]$. 

    Now, consider a solution $\psi:[0,T]\to\real^m\times\real^n$ to the CBF-based dynamics \eqref{eq:interconnection_cbf} for $\alpha\geq\alpha'$, with $\psi(0)=(z_0,x_0)\in\Zset_0\times\Sset$. Note that, denoting $\psi(t) = (\zeta(t),\xi(t))$, from standard CBF theory, as aforementioned, it is guaranteed that $\xi(t)\in\Sset$, for all $t\in[0,T].$ We prove that $\psi(t)\in\Zset\times\Sset$ for all $t\in[0,T]$, thus proving the statement. Reasoning by contradiction, assume that there exists $t_e\in [0,T)$ such that $t_e=\inf \{t:\text{ }\psi(t)\notin\Zset\times\Sset\}$. By absolute continuity of $\psi$, and since $\Zset_0\subseteq\Zset'+\gamma'\ball \subset \Zset'+\gamma\ball\subseteq\Zset$, there also exists time $t'<t_e$ such that $\psi(t') \notin (\Zset'+\gamma'\ball)\times\Sset$ and $\psi(t')\in \Zset\times\Sset$. Thus, from Cor. \ref{cor:interconnected_cbf_solutions_to_perturbed_DI}, the time-truncated solution $\psi|_{[0,t']}$ is also a solution to the perturbed DI \eqref{eq:perturbed_DI}. However, from the above, this implies that $\psi(t') \in (\Zset'+\gamma'\ball)\times\Sset$, which is a contradiction.
\end{proof}

\bibliography{mybib.bib}
\bibliographystyle{IEEEtran}

\end{document}